\documentclass[11pt,a4paper]{article}
\usepackage{a4wide,amsfonts,amsmath,latexsym,amsthm,amssymb,euscript,eufrak,graphicx,units,mathrsfs,setspace,stmaryrd,dsfont}

\usepackage[pdftex]{color}
\definecolor{darkblue}{rgb}{0.3,0.3,0.7}

\usepackage{systeme}
\usepackage{float}
\newfloat{figure}{H}{lof}
\floatname{figure}{\figurename}

\usepackage[french,english]{babel}
\usepackage[T1]{fontenc}
\usepackage[utf8]{inputenc}

\DeclareMathAlphabet{\eufrak}{U}{}{}{}  
\SetMathAlphabet\eufrak{normal}{U}{euf}{m}{n}
\SetMathAlphabet\eufrak{bold}{U}{euf}{b}{n}

\newtheorem{prop}{Proposition}[section]

\newtheorem{theorem}[prop]{Theorem}
\newtheorem{lemma}[prop]{Lemma}
\newtheorem{corollary}[prop]{Corollary}

\newtheorem{assumption}[prop]{Assumption}

\theoremstyle{definition}

\newtheorem{remark}[prop]{Remark}

\newtheorem{definition}[prop]{Definition}
\newtheorem{notation}[prop]{Notation}

\numberwithin{equation}{section}

\def\E{\mathbb{E}}
\def\P{\mathbb{P}}
\def\real{\mathbb{R}}

\def\F{\mathcal{F}}
\def\1{\textbf{1}}

\newcommand{\eps}{\varepsilon}

\newcommand{\be}{\begin{equation}}
\newcommand{\ee}{\end{equation}}
\newcommand{\bde}{\begin{displaymath}}
\newcommand{\ede}{\end{displaymath}}
\newcommand{\beq}{\begin{eqnarray*}}
\newcommand{\eeq}{\end{eqnarray*}}
\newcommand{\beqa}{\begin{eqnarray}}
\newcommand{\eeqa}{\end{eqnarray}}
\newcommand{\bel }{\left\{\begin{array}{ll}}
\newcommand{\eel}{\cr \end{array} \right.}

\newcommand{\bex}{\begin{ex} \rm }
\newcommand{\eex}{\end{ex}}

\def\E{\mathbb E}
\def\F{{\cal F}}

\def\P{\mathbb P}

\def\cal#1{\mathcal{#1}}

\definecolor{ying}{rgb}{0.8, 0.0, 0.04}

\DeclareSymbolFontAlphabet{\mathrsfs}{rsfs}

\author{
Caroline Hillairet\footnote{ENSAE  Paris, CREST,
5  avenue Henry Le Chatelier 91120 Palaiseau, France. \; Email: \texttt{caroline.hillairet@ensae.fr}. \; This research is supported by a grant of the French National Research Agency (ANR), "Investissements d'Avenir" (LabEx Ecodec/ANR-11-LABX-0047) and the Joint Research Initiative "Cyber Risk Insurance: actuarial modeling" with the partnership of AXA Research Fund.} \and Lorick Huang\footnote{INSA de Toulouse, IMT UMR CNRS 5219, Universit\'e de Toulouse, 135 avenue de Rangueil 31077 Toulouse Cedex 4 France. \; Email: \texttt{lorick.huang@insa-toulouse.fr.}} \and Mahmoud Khabou\footnote{INSA de Toulouse, IMT UMR CNRS 5219, Universit\'e de Toulouse, 135 avenue de Rangueil 31077 Toulouse Cedex 4 France. \; Email: \texttt{mahmoud.khabou@insa-toulouse.fr}} \and Anthony R\'eveillac\footnote{INSA de Toulouse, IMT UMR CNRS 5219, Universit\'e de Toulouse, 135 avenue de Rangueil 31077 Toulouse Cedex 4 France. \; Email: \texttt{anthony.reveillac@insa-toulouse.fr}}
}

\title{The Malliavin-Stein method for Hawkes functionals}

\begin{document}

\maketitle

\allowdisplaybreaks

\begin{abstract}
\noindent
In this paper, following  Nourdin-Peccati's methodology, we combine the Malliavin calculus and Stein's method to provide general bounds on the Wasserstein distance between functionals of a compound Hawkes process and a given Gaussian density. To achieve this, we rely on the Poisson embedding representation of an Hawkes process to provide a Malliavin calculus for the Hawkes processes, and more generally for compound Hawkes processes.  As an application, we close a gap in the literature by providing the first Berry-Ess\'een bounds associated to Central Limit Theorems for the compound Hawkes process.  
\end{abstract}


\section{Introduction}
Hawkes processes are an efficient  generalisation of the Poisson processes to model  a sequence of arrivals  over time of some  type of events, that present self-exciting feature, in the sense that each arrival increases the rate of future arrivals for some period of time. This  class of counting processes allows one to capture in a more accurate way, compared to inhomogeneous Poisson processes or Cox processes, self-exciting phenomena.  Introduced by Hawkes in \cite{Hawkes},
there was historically  a first boom in their application in seismology.  Since then they have been widely used in many different fields, among which  neurosciences,  social network models, biology and population dynamics, finance or insurance. In finance,  they are accurate  
 to model  for example credit risk contagion \cite{errais2010affine}, order book or  the microstructure noise's feature of financial markets \cite{bacry2015hawkes}. In insurance also, some risks exhibit self-exciting features, as it is the case for cyber risk \cite{bessy2020multivariate}. \\
Naturally in all these applications, long term behaviour of the Hawkes process $H:=(H_t)_{t \geq 0}$ (or of functionals of $H$) are of interest.  Although we will focus only on (central) limit theorems for Hawkes functionals, we mention that large deviation principles have also been intensively studied in the literature (see \cite{Bordenave_Torrisi} for the first work in that direction for linear Hawkes processes with many extensions in the non-linear case).\\\\
\noindent    
The first main analysis of (central) limit theorems for Hawkes processes has been derived in \cite{Bacry_et_al_2013}. For the sake of the presentation, we recall that a Hawkes process is a counting process $H:=(H_t)_{t \geq 0}$ with stochastic intensity $\lambda:=(\lambda_t)_{t \geq 0}$ given by
$$ \lambda_t = \mu + \int_{(0,t)} \Phi(t-s) dH_s,$$
where $\mu \geq 0$ and $\Phi: \real_+ \to \real_+$ are two given parameters. Under an integrability assumption (which reads as Assumption \ref{assumptionPhi1} below), according to\footnote{Note that we consider a particular case of the results in \cite{Bacry_et_al_2013} which hold for multidimensional Hawkes processes.} \cite[Theorem 2]{Bacry_et_al_2013}, 
\begin{equation}
\label{eq:introconv1}
\left(\frac{H_{Tv} - \int_0^{Tv} \E[\lambda_s] ds}{\sqrt{T}}\right)_{v \in [0,1]} \underset{T\to+\infty}{\overset{\mathcal{L}-\mathcal{S}}{\longrightarrow}} (\tilde \sigma W_v)_{v\in [0,1]},
\end{equation}
where $\mathcal{L}-\mathcal{S}$ stands for the convergence in law as a process in the Skorokhod topology with $\tilde \sigma>0$ an explicit constant depending only on $\mu$ and $\Phi$ and $W$ a Brownian motion. This result is a consequence of the martingale limit theorem that can be found as \cite[Lemma 7]{Bacry_et_al_2013} (once again we present a particular case of this result) 
\begin{equation}
\label{eq:introconv2}
\left(\frac{H_{Tv} - \int_0^{Tv} \lambda_s ds}{\sqrt{T}}\right)_{v \in [0,1]} \underset{T\to+\infty}{\overset{\mathcal{L}-\mathcal{S}}{\longrightarrow}} ({\sigma} W_v)_{v\in [0,1]},
\end{equation}
with $ \sigma>0$ an explicit constant depending only on $\mu$ and $\Phi$ (but still different from $\tilde \sigma$ in (\ref{eq:introconv1})). Naturally, limit theorems (\ref{eq:introconv1}) and (\ref{eq:introconv2}) come as a second step after a law of large numbers for Hawkes processes has been derived; program which has been performed in \cite{Bacry_et_al_2013}.\\\\
\noindent
Generalisations of limit theorems (\ref{eq:introconv1}) and (\ref{eq:introconv2}) have been obtained in \cite{Gao_Zhu}, \cite{Zhu_14}, \cite{Zhu_13}, \cite{Jaisson_Rosenbaum_15}, \cite{Horst_Xu}, \cite{Horst_Wei_2019} for different functionals of the Hawkes process (to mention a few references) in different contexts. \\\\
\noindent 
In spite of this large variety of (functional) limit theorems, it appears that a quantification of this convergence in a form of a Berry-Ess\'een bound \footnote{ Actually, Berry-Ess\'een bounds measure the CLT's rate of convergence by bounding the distance between cumulative distribution functions (Kolmogorov-Smirnov metric). {Here we keep  referring  to the term Berry-Ess\'een for the quantification of  the convergence in the Wasserstein distance}.} is not available in the literature. To be more specific, we are not aware of a Berry-Ess\'een bound for the 1-marginal convergence in law (by taking $v=1$ in (\ref{eq:introconv2})): 
\begin{equation}
\label{eq:introconv3}
\frac{H_{T} - \int_0^{T} \lambda_s ds}{\sqrt{T}} \underset{T\to+\infty}{\overset{\mathcal{L}}{\longrightarrow}} \mathcal{N}(0,{\sigma}^2).
\end{equation}
Several authors (see \cite{Privault:2018aa}, \cite{Besancon} or \cite{torrisi}) have used the Nourdin-Peccati methodology to provide Berry-Ess\'een-type bounds for the normal approximation of counting processes but no result in the literature allows one to quantify convergence (\ref{eq:introconv3}). The closest result is the one of \cite{torrisi}; we refer to Remark \ref{rk:Torrisi} for a discussion on this matter.
\noindent
In this paper, we consider a compound Hawkes process $X:=(X_t)_{t\geq 0}$, of the form 
$$X_t = \sum_{i=1}^{H_t} Y_i, \quad t \geq 0,$$
with $(Y_i)_{i\geq 1}$ a sequence of iid random variables with distribution $\nu$ (and independent of $H$) and  we provide quantitative limit theorems for $X$ in the spirit of (\ref{eq:introconv2})-(\ref{eq:introconv3}). To this end, we adapt the (by now classical) Malliavin-Stein approach (initiated in \cite{Nourdin_Peccati} that we will refer to  Nourdin-Peccati's approach) to derive general bounds in the Wasserstein distance between a class of functionals of the compound Hawkes process and a given Gaussian density (see Theorem \ref{th:main} and especially Relation (\ref{eq:mainestimate})). More precisely, we make use of the Poisson embedding representation of the Hawkes process (initially introduced in \cite{Bremaud_Massoulie}) which allows one to write the Hawkes process $H$ together with its intensity process $\lambda$ and the compound process $X$ as the unique solution of an SDE driven by a Poisson measure $N$ on $\real_+^2\times \real$ (see Theorem \ref{th:HRR} below) to derive a Malliavin calculus with respect to the Hawkes process. This completes the analysis initiated for a general Hawkes process in \cite{Hillairet_Reveillac_Rosenbaum}, where already a Mecke formula has been obtained.  With this material at hand, we can then consider a large class of Hawkes functionals $F$ (taking the form of a divergence operator with respect to the baseline Poisson measure $N$) 
for which we adapt the Malliavin-Stein method of \cite{Nourdin_Peccati} (or more precisely in this Poisson context of \cite{Peccati_2010}) to derive bounds between the Wasserstein distance of any functional $F$ as above and any distribution $\mathcal{N}(0,\sigma^2)$ with $\sigma^2 >0$.  \\\\
\noindent
Our approach includes as a particular case of these functionals, the compound Hawkes process itself (centered) and counterpart of Convergence (\ref{eq:introconv3}) for the Hawkes process $H$ can be quantified in the Wasserstein distance denoted by $d_W$ (see Theorem \ref{th:main2} and Theorem \ref{th:main3}). For instance, in the famous cases of exponential and Erlang's kernels\footnote{The exponential kernel refers to $\Phi(t):=\alpha e^{-\beta t}$, whereas Erlang's kernel refers to $\Phi(t):=\alpha t e^{-\beta t}$ where in both cases $\alpha, \beta$ are well-chosen positive constants.} 
$$ d_W\left(\frac{X_{T} - m \int_0^{T} \lambda_s ds}{\sqrt{T}},  G\right) \leq \frac{C_{\Phi,\nu}}{\sqrt{T}}, \quad \forall T>0, \quad G \sim \mathcal{N}(0,{\sigma}^2 \ \vartheta^2),$$
where $C_{\Phi,\nu}$ and $ \sigma^2, \vartheta^2$ are  explicit constants and $m:=\int_\real x \nu(dx)$ (see Theorem \ref{th:main3} for a precise statement). As Theorem \ref{th:main2} suggests, we were not able to prove the general $T^{-1/2}$ rate for any kernel $\Phi$. Indeed, the term in the bound requires some specific estimates on some cross-correlations of the intensity process of the given Hawkes process that we were not able to perform in a general framework. These correlation estimates are for sure of interest outside the scope of this paper. We refer to Section \ref{section:review} for a review of our results on quantitative counterparts of limit theorems (\ref{eq:introconv2})-(\ref{eq:introconv3}) regarding the (compound) Hawkes process.\\\\
\noindent
The paper is organised as follows. First we derive in Section \ref{section:notations} elements of stochastic analysis for the Hawkes process including the Poisson embedding SDE and the definition of a Malliavin calculus for the Hawkes process (following our companion paper \cite{Hillairet_Reveillac_Rosenbaum}). The main element of the Malliavin-Stein method is also presented in this section. General bounds for Hawkes functionals are presented in Section \ref{section:main} and the Berry-Ess\'een bounds associated to generalisation for the Central Limit Theorem \ref{eq:introconv2} are presented in Section \ref{section:BE}, starting with   a review of our quantitative limit theorems for the compound Hawkes process in Section \ref{section:review}. Technical lemmata including some results on exponential and Erlang kernels are contained in the Appendix (Section \ref{section:teclemma}). 

\section{Notations and preliminaries}
\label{section:notations}

In this section we list all the mathematical framework we will use in our analysis. More precisely, we first recall some general elements of stochastic analysis in Section \ref{sec:Poisson}. Then with this material at hand, we make precise in Section \ref{section:rep} the representation of the compound Hawkes process as a solution to an SDE with respect to a Poisson random measure.  The former representation will turn to be fundamental in our analysis and we provide in Section \ref{section:MalliavinHawkes} the Malliavin derivative of the compound Hawkes process and give an integration by parts formula in Theorem \ref{th:IPPH} which completes the Mecke formula obtained in \cite{Hillairet_Reveillac_Rosenbaum}. Finally, we recall the main elements of Stein's method in Section \ref{section:Stein}.\\\\
For $E$ a topological space, we set $\mathcal B(E)$ the $\sigma$-algebra of Borelian sets. 

\subsection{Elements of stochastic analysis on the Poisson space}
\label{sec:Poisson}
Let $\nu$ be a Borelian measure on $\real$ with $\nu(\real)=1$ and $\nu(\{0\})=0$.\\\\
In this section we model a Hawkes process and a compound Hawkes process using the Poisson embedding representation (or thinning procedure) in the spirit of \cite{Hillairet_Reveillac_Rosenbaum}. To this end, we need three variables for the Poisson measure : $t$ for the jump times, $x$ will stand for the size the jump (whose distribution $\nu$) and $\theta$ will play the role of an auxiliary variable required for the representation itself (according to the thinning algorithm).\\\\   
Let the space of configurations
$$ \Omega^N:=\left\{\omega^N=\sum_{i=1}^{n} \delta_{(t_{i},\theta_i,x_i)}, \; 0=t_0 < t_1 < \cdots < t_n, \; (\theta_i,x_i) \in \real_+\times \real,  \; n\in \mathbb{N}\cup\{+\infty\} \right\}.$$
Each path of a counting process is represented as an element $\omega^N$ in $\Omega^N$ which is a $\mathbb N$-valued measure on $\mathbb{R}_+^2\times \real$. Let $\mathcal F^N$ be the $\sigma$-field associated to the vague topology on $\Omega^N$, and $\P^N$ the Poisson measure under which the counting process $N$ defined as : 
$$ N([0,t]\times[0,b]\times(-\infty,y])(\omega):=\omega([0,t]\times[0,b]\times(-\infty,y]), \quad t \geq 0, \; (b,y) \in \real_+\times \real,$$
is an homogeneous Poisson process with intensity measure $dt \otimes d\theta \otimes \nu$, that is, for any $(t,b,y) \in [0,T]\times \real_+\times \real$, $N([0,t]\times[0,b]\times(-\infty,y])$ is a Poisson random variable with intensity $ b \, t \, \nu((-\infty,y])$.\\\\ 
We set $\mathbb F^N:=(\F_t^N)_{t\geq 0}$ the natural history of $N$, that is $\mathcal{F}_t^N:=\sigma(N( \mathcal T  \times B), \; \mathcal T \subset \mathcal{B}([0,t]), \; B \in \mathcal{B}(\real_+\times \real))$.
Let also, $\mathcal F_\infty^N:=\lim_{t\to+\infty} \mathcal F_t^N$. The expectation with respect to $\P^N$ is denoted by $\E[\cdot]$. For $t\geq 0$, we denote by $\E_t[\cdot]$ the conditional expectation $\E[\cdot \vert \mathcal F_t^N]$.\\\\
\noindent
We describe some elements of stochastic analysis on the Poisson space, especially the shift operator, the Malliavin derivative and its dual operator : the divergence. 

\begin{definition}[Shift operator]
\label{definition:shift}
We define for $(t,\theta,x)$ in $\mathbb{R}_+\times \real_+\times \real$ the measurable maps  
$$ 
\begin{array}{lll}
\eps_{(t,\theta,x)}^+ : &\Omega^N &\to \Omega^N\\
&\omega &\mapsto  \eps_{(t,\theta,x)}^+(\omega),
\end{array}
$$
where for any $A$ in $\mathcal{B}(\mathbb{R}_+\times \real_+\times \real)$
$$(\eps_{(t,\theta,x)}^+(\omega))(A) := \omega(A \setminus {(t,\theta,x)}) + \textbf{1}_A(t,\theta,x),$$
with 
$$ \textbf{1}_{A}(t,\theta,x):=\left\lbrace \begin{array}{l} 1, \quad \textrm{if } (t,\theta,x)\in A,\\0, \quad \textrm{else.}\end{array}\right. $$
\end{definition}

\begin{lemma}
\label{lemma:mesur}
Let $t \geq 0$ and $F$ be an $\mathcal F_t^N$-measurable random variable. Let $v > t$ and $(\theta,x)\in \real_+\times \real$. It holds that 
$$ F\circ\eps_{(v,\theta,x)}^+ = F, \quad \P-a.s.. $$
\end{lemma}

\begin{definition}[Malliavin derivative]
For $F$ in $L^2(\Omega,\mathcal F_\infty^N,\P)$, we define $D F$ the Malliavin derivative of $F$ as 
$$ D_{(t,\theta,x)} F := F\circ \eps_{(t,\theta,x)}^+ - F, \quad (t,\theta,x) \in \real_+^2\times \real.  $$
\end{definition}

\begin{remark}
{Our approach  follows \cite{Picard_French_96,Privault_LectureNotes}. In our setting, the measure $\nu$ induces a Malliavin calculus for the compound Poisson process as a L\'evy process. In that realm, the previous Malliavin calculus was extended to general L\'evy processes in \cite{Sole_etal}. There the Malliavin derivative takes the form of : $D_{(t,\theta,x)} F:= \frac{F\circ \eps^+_{(t,\theta,x)} - F}{x}$ but the intensity measure is not exactly the one we consider here. Hence, here we decided to follow the classical approach (without the $1/x$ normalisation) with compensator $\nu$ (we refer to \cite{Murr} or \cite[Section 6.7]{Privault_LectureNotes} for more details) as  our baseline process is the Poisson process $N$ and not the compound process $\int x N(dt,d\theta,dx)$.}
\end{remark}

\noindent The following definition is a by-product of \cite[Th\'eor\`eme 1]{Picard_French_96} (see also \cite{nualart1990anticipative}).
\begin{definition}
Let $\mathcal I$  be the sub-sigma field of $\mathcal{B}(\real_+^2\times \real)\otimes \mathcal F^N$   of stochastic processes $Z:=(Z_{(t,\theta,x)})_{(t,\theta,x) \in \real_+\times \real_+\times \real}$ in $L^1(\Omega \times \real_+^2\times \real),\P\otimes dt\otimes d\theta \otimes \nu)$  such that 
$$ D_{(t,\theta,x)} Z_{(t,\theta,x)} = 0, \quad \textrm{ for a.a. } (t,\theta,x) \in \real^2_+ \times \real.$$
\end{definition}

\begin{remark}
\label{rem:Nshift}
Let $(t_0,\theta_0,x_0)$ in $\real^2_+\times \real$, $(s,t)$ in $\real_+^2$ with $t_0 < s < t$. For $\mathcal T \in \{(s,t), (s,t], [s,t), [s,t]\}$ and $B$ in $\mathcal B(\real_+^2 \times \real)$,  we have that : 
$$ N\circ \eps_{(t_0,\theta_0,x_0)}^+ (\mathcal T\times B) = N (\mathcal T\times B). $$
\end{remark}

\begin{definition}
We set $\mathcal S$ the set of stochastic processes $Z:=(Z_{(t,\theta,x)})_{(t,\theta,x) \in \real_+^2\times \real}$ in $\mathcal I$ such that : 
$$ \E\left[\int_{\real_+^2\times \real} \left|Z_{(t,\theta,x)}\right|^2 dt d\theta \nu(dx)\right] + \E\left[\left(\int_{\real_+^2\times \real} Z_{(t,\theta,x)} N(dt,d\theta,dx)\right)^2\right]<+\infty,$$
where $\int_{\real_+^2\times \real} Z_{(t,\theta,x)} N(dt,d\theta,dx)$  is understood in the sense of the Stieltjes integral.\\\\ 
For $Z$ in $\mathcal S$, we set the divergence operator with respect to $N$ as  
\begin{equation}
\label{eq:delta}
\delta^N(Z):=\int_{\real_+^2\times \real} Z_{(t,\theta,x)} N(dt,d\theta,dx) - \int_{\real_+^2\times \real} Z_{(t,\theta,x)} dt d\theta \nu(dx).
\end{equation}
\end{definition}

\noindent We conclude this section with the integration by parts formula on the Poisson space (see \cite[Remarque 1]{Picard_French_96})
\begin{prop}[See \textit{e.g.} \cite{Picard_French_96}]
\label{prop:IPP}
Let $F$ be in $L^2(\Omega,\mathcal F_\infty^N,\P)$ and $Z$ be in $\mathcal{S}$. We have that 
\begin{equation}
\label{eq:IBPPoisson}
\E\left[F \delta^N(Z)\right] = \E\left[\int_{\real_+^2\times \real} Z_{(t,\theta,x)} D_{(t,\theta,x)} F dt d\theta \nu(dx)\right].
\end{equation}
\end{prop} 
 
\subsection{Representation of the compound Hawkes process}
\label{section:rep}

We first recall the definition of a Hawkes process. 

\begin{definition}[Standard Hawkes process, \cite{Hawkes}]
\label{def:standardHawkes}
Let $\mu>0$ and $\Phi:\real_+ \to \real_+$ be a bounded non-negative map with $\|\Phi\|_1 :=\int_0^{+\infty} \Phi(u) du<1$. A standard Hawkes process $H:=(H_t)_{t \geq 0}$ with parameters $\mu$ and $\Phi$ is a counting process such that   
\begin{itemize}
\item[(i)] $H_0=0,\quad \P-a.s.$,
\item[(ii)] its ($\mathbb{F}^N$-predictable) intensity process is given by
$$\lambda_t:=\mu + \int_{(0,t)} \Phi(t-s) dH_s, \quad t\geq 0,$$
that is for any $0\leq s \leq t $ and $A \in \mathcal{F}^N_s$,
$$ \E\left[\textbf{1}_A (H_t-H_s) \right] = \E\left[\int_{(s,t]} \textbf{1}_A \lambda_r dr \right].$$
\end{itemize}
\end{definition}
\noindent This definition can be generalized as follows. 

\begin{definition}[Generalized Hawkes process]
\label{def:Hawkes}
Let $v \geq 0$, $h^v$ be a $\mathcal F_v^N$-measurable random variable with valued in $\mathbb N$, $\mu^v:=(\mu^v(t))_{t\geq v}$ a positive map such that $\mu^v(t)$ is $\mathcal F_v^N$-measurable for any $t\geq v$, and $\Phi:\real_+ \to \real_+$ be a bounded non-negative map with $\|\Phi\|_1 <1$. A Hawkes process on $[v,+\infty)$ with parameters $\mu^v$, $h^v$ and $\Phi:\real_+ \to \real_+$ is a ($\mathbb{F}^N$-adapted) counting process $H:=(H_t)_{t\geq v}$ such that 
\begin{itemize}
\item[(i)] $H_v=h^v,\quad \P-a.s.$,
\item[(ii)] its ($\mathbb{F}^N$-predictable) intensity process is given by
$$\lambda_t:=\mu^v(t) + \int_{(v,t)} \Phi(t-s) dH_s, \quad t\geq v,$$
that is for any $v\leq s \leq t $ and $A \in \mathcal{F}^N_s$, 
$$ \E\left[\textbf{1}_A (H_t-H_s) \vert \mathcal F_v\right] = \E\left[\int_{(s,t]} \textbf{1}_A \lambda_r dr \Big\vert \mathcal F_v \right].$$
\end{itemize}
\end{definition}

\begin{definition}[Compound Hawkes process]
\label{definition:compound}
Let $\mu>0$, $\Phi:\real_+ \to \real_+$ be a bounded non-negative map with $\|\Phi\|_1 <1$ and $(Y_i)_{i\geq 1}$ be iid random variables, with common distribution $\nu$, and independent of the Hawkes process $H$ with parameters $\mu$ and $\Phi$. We name compound Hawkes process  $X:=(X_t)_{t \geq 0}$ a stochastic process with representation : 
\begin{equation}
\label{eq:compound}
X_t = \sum_{i=1}^{H_t} Y_i, \quad t \geq 0.
\end{equation}
\end{definition}
\noindent  
We now represent a Hawkes process (and a compound Hawkes process) as the unique solution to an SDE driven by $N$. This representation relies on the "Poisson embedding" (or "Thinning Algorithm" (see \textit{e.g.} \cite{Bremaud_Massoulie,Costa_etal,Daley_VereJones,Ogata} and references therein). The next result is an extension of \cite[Theorem 3.3]{Hillairet_Reveillac_Rosenbaum}.

\begin{theorem}[See \cite{Hillairet_Reveillac_Rosenbaum}]
\label{th:HRR}
Let $\mu >0$ and $\Phi:\real_+ \to \real_+$ such that $\|\Phi\|_1<1$. The SDE below admits a unique\footnote{we refer to \cite[Theorem 3.3]{Hillairet_Reveillac_Rosenbaum} for a precise statement about uniqueness} solution $(X,H,\lambda)$ with $H$ (resp. $\lambda$) $\mathbb F^N$-adapted (resp. $\mathbb F^N$-predictable) 
\begin{equation}
\label{eq:H}
\left\lbrace
\begin{array}{l}
X_t = \displaystyle{\int_{(0,t]\times \real_+\times \real} x \textbf{1}_{\{\theta \leq \lambda_s\}} N(ds,d\theta,dx)}, \quad t \geq 0, \\\\
H_t = \displaystyle{\int_{(0,t]\times \real_+\times \real} \textbf{1}_{\{\theta \leq \lambda_s\}} N(ds,d\theta,dx)},\quad t \geq 0, \\\\
\lambda_t = \mu + \int_{(0,t)} \Phi(t-u) d H_u,\quad t \geq 0 .
\end{array}
\right.
\end{equation}
In addition, $(H,\lambda)$ is a Hawkes process in the sense of Definition \ref{def:standardHawkes}. 
We set $\mathbb F^H:=(\mathcal F_t^H)_{t \geq 0}$ (respectively $\mathbb F^X:=(\mathcal F_t^X)_{t \geq 0}$) the natural filtration of $H$ (respectively of $X$) and $\mathcal F_\infty^H:=\lim_{t\to+\infty} \mathcal F_t^H$ (respectively $\mathcal F_\infty^X:=\lim_{t\to+\infty} \mathcal F_t^X$).  Obviously $\mathcal F_t^H \subset \mathcal F_t^X \subset \mathcal F_t^N$ as $H$ is completely determined by the jump times of $H$ which are exactly those of $X$.\\
Finally, $X$ is a compound Hawkes process in the sense of Definition \ref{definition:compound}.
 \end{theorem}

\begin{remark}
The definition above is based on the following notation. Indeed, the rigorous definition of $H$ is given as : 
$$ H(\mathcal T):=\int_{\real_+\times \real_+\times \real} \textbf{1}_{\{s \in \mathcal T\}} \textbf{1}_{\{\theta \leq \lambda_s\}} N(ds,d\theta,dx),\quad \mathcal T \in \mathcal B(\real_+),$$
$$ X(\mathcal T):=\int_{\real_+\times \real_+\times \real} x \textbf{1}_{\{s \in \mathcal T\}} \textbf{1}_{\{\theta \leq \lambda_s\}} N(ds,d\theta,dx),\quad \mathcal T \in \mathcal B(\real_+).$$
Then by convention we set  $H_t := H([0,t])$, $X_t:=X([0,t])$ and $H_{t-}:=H([0,t))$, $X_{t-}:=X([0,t))$ for any $t >0$, with $H_0:=X_0:=0$.
\end{remark}

\begin{remark}
{We have decided to include in our analysis the case of compound Hawkes processes and not simply the one of Hawkes processes. By choosing $\nu(dx) = \delta_{1}(dx)$ (the Dirac measure concentrated at $x=1)$ one obviously recover the construction of the Hawkes process as in \cite{Hillairet_Reveillac_Rosenbaum} and in that case one can just consider $N$ to be a Poisson measure on $\real_+^2$.}
\end{remark}

\subsection{Malliavin analysis of the compound Hawkes process}
\label{section:MalliavinHawkes}

We now wish to describe the impact of the Malliavin derivative on the Hawkes process. Once again this material relies on the one provided in \cite{Hillairet_Reveillac_Rosenbaum}. The objective of this section is to derive an integration by parts formula for the Hawkes functionals as Theorem \ref{th:IPPH}.

\begin{lemma}
\label{lemma:TempDH}
Let $t$ and $v$ in $\real_+$, $(\theta,x)$ and $ (\theta_0,x_0)$ in $\real_+\times \real$, it holds that :
\begin{align*}
& \quad \textbf{1}_{\{\theta \leq \lambda_t\}} (X_v\circ \eps_{(t,\theta,x)}^+,H_v\circ \eps_{(t,\theta,x)}^+,\lambda_v\circ \eps_{(t,\theta,x)}^+)_{v\geq 0} \\
&= \textbf{1}_{\{\theta_0 \leq \lambda_t\}} (X_v\circ \eps_{(t,\theta_0,x)}^+,H_v\circ \eps_{(t,\theta_0,x_0)}^+,\lambda_v\circ \eps_{(t,\theta_0,x_0)}^+)_{v\geq 0}.
\end{align*}
\end{lemma}

\begin{proof}
The proof relies on the very definition of the shift operator $\eps_{(t,\theta,x)}^+$ and on the structure of the Hawkes process. Indeed, the intensity is impacted (on $(t,+\infty)$) by the addition of a jump to $H$ at time $t$, but the value of this jump for $H$ is equal to $1$, and this  for any $\theta$ such that $\theta \leq \lambda_t$. In other words : for $v \geq t$ on $\{\theta \leq \lambda_t\}$
\begin{align*}
H_v \circ \eps_{(t,\theta_0,x_0)}^+ &=H_{t-} + \left(\int_{[t,v]\times \real_+\times \real} \textbf{1}_{\{\theta \leq \lambda_u\}} N(d\theta,du,dy)\right) \circ \eps_{(t,\theta_0,x_0)}^+ \\
&{=H_{t-} + \left(\int_{(t,v]\times \real_+\times \real} \textbf{1}_{\{\theta \leq \lambda_u\}} N(d\theta,du,dy)\right) \circ \eps_{(t,\theta_0,x_0)}^+} \\
&=H_{t-} + \textbf{1}_{\{\theta_0 \leq \lambda_t\}} + \int_{(t,v]\times \times \real} \int_{\real_+} \textbf{1}_{\{\theta \leq \lambda_u \circ \eps_{(t,\theta_0,x_0)}^+\}} N(d\theta,du,dy),
\end{align*}
and $H_v \circ \eps_{(t,\theta_0,x_0)}^+ = H_v$ for any $v<t$. Note that it is part of the definition of the shift operator (see Definition \ref{definition:shift}) to remove the possible natural jump at time $t$, which explain why we move from the integral $\int_{[t,v]}$ to $\int_{(t,v]}$ in the computations above. Similarly,  {for $v \geq t$ }
\begin{align*}
X_v \circ \eps_{(t,\theta_0,x_0)}^+ &=X_{t-} + \left(\int_{(t,v]\times \real_+\times \real} y \textbf{1}_{\{\theta \leq \lambda_u\}} N(d\theta,du,dy)\right) \circ \eps_{(t,\theta_0,x_0)}^+ \\
&=X_{t-} + x_0 \textbf{1}_{\{\theta_0 \leq \lambda_t\}} + \int_{(t,v] \times \real} \int_{\real_+} y \textbf{1}_{\{\theta \leq \lambda_u \circ \eps_{(t,\theta_0,x_0)}^+\}} N(d\theta,du,dy),
\end{align*}
and $X_v \circ \eps_{(t,\theta_0,x_0)}^+ = X_v$ for any $v<t$. In a similar fashion,
$\lambda_v \circ \eps_{(t,\theta_0,x_0)}^+ = \lambda_v$ for any $v \leq t$ and for $v>t$,
\begin{align*}
\lambda_v \circ \eps_{(t,\theta_0,x_0)}^+ 
&=\left(\mu + \int_{(0,t)} \Phi(v-u) d H_u + \int_{[t,v)} \Phi(v-u) d H_u\right) \circ \eps_{(t,\theta_0,x_0)}^+ \\
&=\mu + \int_{(0,t)} \Phi(v-u) d H_u + \Phi(v-t) + \int_{(t,v)} \Phi(v-u) d (H_u\circ \eps_{(t,\theta_0,x_0)}^+).
\end{align*}
In other words, $(X \circ \eps_{(v,\theta_0,x_0)}^+,H \circ \eps_{(v,\theta_0,x_0)}^+,\lambda \circ \eps_{(v,\theta_0,x_0)}^+)$ solves the same (pathwise and in the SDE sense) equation for any $\theta_0$ such that $\theta_0 \leq \lambda_t$.
\end{proof}

\begin{prop}
\label{prop:TempDH}
Let $F$ be a $\mathcal F_\infty^X$-measurable random variable. Then for any $t\geq 0$, and $\forall \theta, \theta_0\geq 0$, $\forall x \in \real$,
$$ \textbf{1}_{\{\theta \leq \lambda_t\}} D_{(t,\theta,x)} F = \textbf{1}_{\{\theta_0 \leq \lambda_t\}} D_{(t,\theta_0,x)} F, \quad \P-a.s..$$
\end{prop}

\begin{proof}
This follows from Lemma \ref{lemma:TempDH} and the fact that any $\mathcal F_T^X$-measurable random variable $F$ is a limit of a random variable of the form $\varphi((X_{t_1},H_{t_1},\lambda_{t_1}),\cdots,(X_{t_n},H_{t_n},\lambda_{t_n}))$, with : $0\leq t_1<\ldots<t_n$, $n\geq 1$ and $\varphi$ a Borelian map from $\real^{3n}$ to $\real$. More precisely, $H$ is itself a functional of $X$ as it is a counting process whose jumps coincide with those of $X$. But the fact of distinguishing $X$ and $H$ allows one to note for instance that the Malliavin derivative of a functional of $H$ does not depend on the $x$ variable for instance.  
\end{proof}
\noindent This motivates the introduction of the following notation. 
\begin{definition}
\label{def:DH}
Let $t\geq 0$. For a $\mathcal F_\infty^X$-measurable  random variable $F$, we set 
$$ D_{(t,\lambda_t,x)} F:= \textbf{1}_{\{\theta \leq \lambda_t\}} D_{(t,\theta,x)} F, \quad \forall (\theta,x) \in \real_+\times \real.$$
\end{definition}
\noindent As a consequence, if  $(F_s)_{s \geq 0}$ is a $\mathbb F^X$-measurable (resp. predictable) process, then   $D_{(t,\lambda_t,x)} F_s=0$ for $s<t$ (resp. for $s\leq t$)  since $D_{(t,\lambda_t,x)} F_s = F_s\circ \eps_{(t,\lambda_t,x)}^+ - F_s$  and using   Lemma \ref{lemma:mesur}. 

\begin{remark}
Let $F$ be a functional of $(H,\lambda)$. Then for any $t\geq 0$, and $\forall \theta, \theta_0\geq 0$, $\forall x, y \in \real$,
$$ \textbf{1}_{\{\theta \leq \lambda_t\}} D_{(t,\theta,x)} F = \textbf{1}_{\{\theta_0 \leq \lambda_t\}} D_{(t,\theta_0,y)} F, \quad \P-a.s..$$
In other words, in that case, the Malliavin derivative does not depend on the variable $x$.
\end{remark}
\noindent With this notation at hand, we determine the Malliavin derivative of the compound Hawkes process. 
\begin{prop}
\label{prop:DescDH}
Let $t\geq 0$ and $x\in \real$. We have
\begin{equation*}
(D_{(t,\lambda_t,x)} X_s,D_{(t,\lambda_t,x)} H_s,D_{(t,\lambda_t,x)} \lambda_s)  = \left\lbrace
\begin{array}{l}
(x+\hat X_s^t,1+\hat H_s^t,\hat \lambda_s^t),   \quad   s\geq t, \\\\
(0,0,0), \quad \quad \quad \quad  s<t
\end{array}
\right.
\end{equation*}
where the equality is understood pathwise and in the SDE sense and 
where $(\hat X_s^t,\hat H_s^t,\hat \lambda_s^t)_{s\geq t}$ is the unique solution to the SDE 
\begin{equation}
\label{eq:DH}
\left\lbrace
\begin{array}{l}
\hat H_s^{t} = \displaystyle{\int_{(t,s]\times \real_+\times \real} \textbf{1}_{\{\lambda_u \leq \theta \leq \lambda_u+ \hat \lambda_u^{t}\}} N(du,d\theta,dy)}, \quad s \geq t, \\\\
\hat X_s^{t} = \displaystyle{\int_{(t,s]\times \real_+\times \real} y \textbf{1}_{\{\lambda_u \leq \theta \leq \lambda_u+ \hat \lambda_u^{t}\}} N(du,d\theta,dy)}, \quad s \geq t,\\\\
\hat \lambda_s^{t} =  \Phi(s-t) + \displaystyle{\int_{(t,s)} \Phi(s-u) d\hat H_u^{t}}, \quad s>t, \; \hat \lambda_t^{t}=0.
\end{array}
\right.
\end{equation}
In addition, $(\hat H_s^t,\hat \lambda_s^t)_{ s \in [t,+\infty)}$ is a generalized Hawkes process, with initial intensity that is not bounded away from $0$.
\end{prop}

\begin{proof}
Fix $t \geq 0$ and $x\in \real$. According to Definition \ref{def:DH}, $(D_{(t,\lambda_t,x)} X,D_{(t,\lambda_t,x)} H,D_{(t,\lambda_t,x)} \lambda)$ is given as $D_{(t,\lambda_t,x)} X_s = X_s\circ \eps_{(t,\lambda_t,x)}^+ - X_s$, $D_{(t,\lambda_t,x)} H_s = H_s\circ \eps_{(t,\lambda_t,x)}^+ - H_s$, $D_{(t,\lambda_t,x)} \lambda_s = \lambda_s\circ \eps_{(t,\lambda_t,x)}^+ - \lambda_s$. Besides,  $D_{(t,\lambda_t,x)} X_s=D_{(t,\lambda_t,x)} H_s =0$ for $s<t$ and $D_{(t,\lambda_t,x)} \lambda_s =0$ for $s\leq t$. Let $s\geq t$, according to Lemma \ref{lemma:TempDH} we have that 
\begin{align*}
&D_{(t,\lambda_t,x)} H_s \\
&= H_{t-} + 1 + \int_{(t,s]\times \real_+\times \real} \textbf{1}_{\{\theta \leq \lambda_u \circ \eps_{(t,\lambda_t,x)}^+\}} N(d\theta,du,dy) - \left(H_t + \int_{(t,s]\times \real_+\times \real} \textbf{1}_{\{\theta \leq \lambda_u \}} N(d\theta,du,dy)\right) \\
&= -(H_t-H_{t-}) +1 + \int_{(t,s] \times \real_+\times \real} \textbf{1}_{\{\lambda_u \leq \theta \leq \lambda_u \circ \eps_{(t,\lambda_t,x)}^+\}} N(d\theta,du,dy) \\
&= 1 + \int_{(t,s]\times \real_+\times \real} \textbf{1}_{\{\lambda_u \leq \theta \leq \lambda_u \circ \eps_{(t,\lambda_t,x)}^+\}} N(d\theta,du,dy), \quad \P-a.s..
\end{align*}
\begin{align*}
& D_{(t,\lambda_t,x)} X_s \\
&= X_{t-} + x + \int_{(t,s]\times \real_+\times \real} \textbf{1}_{\{\theta \leq \lambda_u \circ \eps_{(t,\lambda_t,x)}^+\}} y N(d\theta,du,dy) - \left(X_t + \int_{(t,s]\times \real_+\times \real} \textbf{1}_{\{\theta \leq \lambda_u \}} y N(d\theta,du,dy)\right) \\
&= -(X_t-X_{t-}) + x + \int_{(t,s] \times \real_+\times \real} \textbf{1}_{\{\lambda_u \leq \theta \leq \lambda_u \circ \eps_{(t,\lambda_t,x)}^+\}} y N(d\theta,du,dy) \\
&= x + \int_{(t,s]\times \real_+\times \real} \textbf{1}_{\{\lambda_u \leq \theta \leq \lambda_u \circ \eps_{(t,\lambda_t,x)}^+\}} y N(d\theta,du,dy), \quad \P-a.s..
\end{align*}
Note that we ignore the possible jump at time $t$ of $H$, as it is equal to $0$ $\P$-a.s.. This does not lead to an issue when integrating in $t$ as $\P\left[\{u, \; \Delta_u H \neq 0\}<+\infty\right]=1$ (the same comment apply to $X$). More generally, the Malliavin derivative $DF$ coincides $\P\otimes (dt\otimes d\theta \otimes \nu)$-a.e. with the operator $F \circ \eps_{(t,\theta,x)}^+ - F$ where it is understood that we exclude the possible jump of $N$ at time $t$ which is supported by a set of probability $0$. Similarly, for $s>t$, we have that 
\begin{align*}
D_{(t,\lambda_t,x)} \lambda_s 
&= \Phi(s-t) + \int_{(t,s)} \Phi(s-u) d (H_u\circ \eps_{(t,\lambda_t,x)}^+ - dH_u) \\
&= \Phi(s-t) + \int_{(t,s)} \Phi(s-u) d (D_{(t,\lambda_t,x)} H_u).
\end{align*}
For $s=t$, $D_{(t,\lambda_t,x)} \lambda_s=0$ as $\lambda$ is $\mathbb F^H$-predictable. 
Writing $\lambda_u \circ \eps_{(t,\lambda_t,x)}^+ = \lambda_u + D_{(t,\lambda_t,x)} \lambda_u$ proves that $(D_{(t,\lambda_t,x)} X,D_{(t,\lambda_t,x)} H,D_{(t,\lambda_t,x)} \lambda)$ is solution to SDE (\ref{eq:DH}) which admits a unique solution following \cite[Theorem 3.3]{Hillairet_Reveillac_Rosenbaum}.\\
The last claim follows by proving that $(\hat H^t,\hat \lambda^t)$ is a Hawkes process on $[t,+\infty)$. Indeed, for $t \leq s_1 \leq s_2$, we have that : 
\begin{align*}
\E\left[\hat H_{s_2}^t - \hat H_{s_1}^t\vert \mathcal F_{s_1}^H\right]
&=\E\left[\E\left[\hat H_{s_2}^t - \hat H_{s_1}^t\vert \mathcal F_{s_1}^N\right]\vert \mathcal F_{s_1}^H\right] \\
&=\E\left[\int_{s_1}^{s_2} \E\left[ \int_{\real_+} \textbf{1}_{\{\lambda_u \leq \theta \leq \lambda_u + \hat \lambda^t_u \}} d\theta \vert \mathcal F_{s_1}^N\right]  dt\vert \mathcal F_{s_1}^H\right] \\
&=\int_{s_1}^{s_2} \E\left[ \hat \lambda^t_u \vert \mathcal F_{s_1}^H\right] dt.
\end{align*}
Note that the intensity  $(\hat \lambda_s^t)_{ s \in [t,+\infty)}$ can reach zero.
\end{proof}

\noindent We conclude this section by re-writing the integration by parts formula (\ref{eq:IBPPoisson}) for the Hawkes process.

\begin{theorem}
\label{th:IPPH}
Set $\mathcal Z:=(\mathcal Z_{(t,\theta)})_{(t,\theta) \in \real^2_+}$ the stochastic process defined as 
$$ \mathcal Z_{(t,\theta)}:= \textbf{1}_{\{\theta \leq \lambda_t\}}, \quad (t,\theta) \in \real_+^2.$$
Let $Z:=(Z_{(t,x)})_{(t,x) \in \real_+ \times \real}$ be a $\mathbb F^X$-predictable process satisfying
$$\E\left[\int_{\real_+ \times \real} |Z_{(t,x)}|^2 \lambda_t dt \nu(dx) + \left(\int_{\real_+ \times \real} Z_{(t,x)} \lambda_t dt \nu(dx)\right)^2\right]<\infty.$$ It holds that 
\begin{itemize}
\item[(i)] $Z \mathcal Z=(Z_{(t,x)} \textbf{1}_{\{\theta \leq \lambda_t\}})_{(t,\theta,x)\in \real_+^2\times \real}$ belongs to $\mathcal S$.
\item[(ii)] For any $\mathcal F^X_\infty$-measurable random variable $F$ with $\E[|F|^2]<+\infty$,
\begin{equation}
\label{eq:IBP}
\E\left[F \delta^N(Z \textbf{1}_{\{\theta \leq \lambda_t\}})\right] = \E\left[ \int_{\real_+ \times \real} \lambda_t Z_{(t,x)} D_{(t,\lambda_t,x)} F dt \nu(dx)\right],
\end{equation}
where $D_{(t,\lambda_t,x)} F = D_{(t,\theta,x)} F \textbf{1}_{\{\theta \leq \lambda_t\}}$ for any $\theta \geq 0$, $x\in \real$.
\end{itemize}
\end{theorem}

\begin{proof}
By construction, $\lambda$ is $\mathbb F^H$-predictable. Hence, for any $(t,\theta,x)$ in $\real_+^2\times \real$, $D_{(t,\theta,x)} \textbf{1}_{\{\theta \leq \lambda_t\}} =0$. So, $\mathcal Z$ belongs to $\mathcal I$. In addition, as $Z$ and $\lambda$ are predictable, our assumptions imply that 
\begin{align*}
&\E\left[\left(\int_{\real_+^2\times \real} Z_{(t,x)} \mathcal Z_{(t,\theta)} N(dt,d\theta,dx)\right)^2\right]\\
&\leq 2 \left(\E\left[\left(\int_{\real_+^2\times \real} Z_{(t,x)} \mathcal Z_{(t,\theta)} (N(dt,d\theta,dx)-dt d\theta \nu(dx))\right)^2\right]\right. \\
&\quad + \left.\E\left[\left(\int_{(\real_+^2\times \real} Z_{(t,x)} \mathcal Z_{(t,\theta)}  dt d\theta \nu(dx)\right)^2\right] \right) \\
&=  2 \left(\E\left[\int_{\real_+^2\times \real} |Z_{(t,x)} \mathcal Z_{(t,\theta)}|^2 dt d\theta \nu(dx)\right] + \E\left[\left(\int_{\real_+^2\times \real} Z_{(t,x)} \lambda_t dt \nu(dx)\right)^2\right]  \right)\\
&= 2 \left(\E\left[\int_{\real_+\times \real}  |Z_{(t,x)}|^2  \lambda_t dt \nu(dx) \right] + \E\left[\left(\int_{\real_+\times \real} Z_{(t,x)} \lambda_t dt \nu(dx)\right)^2\right]  \right)<+\infty
\end{align*}
and 
$$\E\left[\int_{0}^{+\infty} \int_{\real_+\times \real} \left|Z_{(t,x)} \mathcal Z_{(t,\theta)} \right|^2 d\theta  dt \nu(dx)\right] 
=\E\left[\int_{0}^{+\infty} \int_\real |Z_{(t,x)}|^2 \lambda_t dt  \nu(dx) \right] <+\infty.$$
This proves (i). In particular, $\delta^N(Z \mathcal Z)$ is well-defined and the IBP formula (\ref{eq:IBPPoisson}) is in force. It gives that 
\begin{align*}
\E\left[F \delta^N(Z \mathcal Z)\right] &= \E\left[\int_{\real_+^2\times \real} Z_{(t,x)} \textbf{1}_{\{\theta \leq \lambda_t\}} D_{(t,\theta,x)} F d\theta dt \nu(dx)\right]\\
&= \E\left[\int_0^{+\infty} \int_\real Z_{(t,x)} \int_{\real_+} \textbf{1}_{\{\theta \leq \lambda_t\}} D_{(t,\theta,x)} F d\theta \nu(dx) dt\right] \\
&= \E\left[\int_0^{+\infty} \int_\real Z_{(t,x)} \int_0^{\lambda_t} D_{(t,\lambda_t,x)} F d\theta \nu(dx) dt\right] \\
&= \E\left[\int_0^{+\infty} \lambda_t \int_\real Z_{(t,x)} D_{(t,\lambda_t,x)} F \nu(dx) dt \right].
\end{align*}
\end{proof} 

\subsection{Elements on Stein's method}
\label{section:Stein}

Whereas Stein's method has been introduced by C. M. Stein in \cite{Stein_method}, the combination of the Malliavin calculus with Stein's method (and known as the Nourdin-Peccati's approach) has been initiated in \cite{Nourdin_Peccati} for the approximation of Gaussian functionals, extended in \cite{Peccati_2010} for Poisson functionals (which is closer to our paper). We introduce in this section the original Stein's approach which allows one to derive Inequality (\ref{eq:Stein}) below. Then in the proof of Theorem \ref{th:main} we will adapt \cite{Peccati_2010} to transform the right-hand side of (\ref{eq:Stein}) into the right-hand side of (\ref{eq:mainestimate}) using the Malliavin calculus. We refer to \cite{NourdinPeccati_Book} for a complete exposition of the original Stein method and of the Nourdin-Peccati approach.

\begin{definition}
\label{def:W}
Let $F$ and $G$ two random variables defined on some $(\Omega, \mathcal F_\infty^N,\P)$. We recall that the Wasserstein distance between $\mathcal{L}_F$ and $\mathcal{L}_G$ (or simply between $F$ and $G$) as : 
$$ d_W(F,G) :=\sup_{h \in \textrm{Lip}} \left| \E[h(F)]-\E[h(G)] \right|,$$
with $\textrm{Lip}:=\left\{h:\real \to \real \textrm{ differentiable a.e. with }\; \|h'\|_\infty \leq 1\right\}$. 
\end{definition}
\noindent
We assume that $F$ is centered. Let $G \sim \mathcal{N}(0,\sigma^2)$. 
We set 
$$\mathcal F_W^0:=\left\{f:\real \to \real,\; \textrm{ twice differentiable with } \|f'\|_\infty \leq 1, \; \|f''\|_\infty \leq 2,\; f(0)=0 \right\}.$$
\noindent
Consider $h$ in $\textrm{Lip}$. C. M. Stein proved in \cite{Stein_method}, that there exists a function $f_h$ in $\mathcal{F}_W^0$ solution to the functional equation (named Stein's equation) :
$$ h(x)-\E[h(G)] = \sigma^2 f_h'(x) - x f_h(x), \quad \forall x \in \real. $$
Plugging $F$ in this equation and taking the expectation, we get that : 
$$ \left|\E[h(F)]-\E[h(G)]\right| = \left|\E[\sigma^2 f_h'(F) - F f_h(F)]\right|.$$
Hence, 
\begin{equation}
\label{eq:Stein}
d_W(F,G) \leq \sup_{f \in \mathcal{F}_W^0} \left|\E[\sigma^2 f'(F) - F f(F)]\right|.
\end{equation}
In addition, the right hand side is equal to $0$ if and only if $F\sim \mathcal{N}(0,\sigma^2)$.
\begin{remark}
	The original result proven by C. M. Stein in \cite{Stein_method} was the existence of a function $f_h$ in 
	$$\mathcal F_W:=\left\{f:\real \to \real,\; \textrm{ twice differentiable with } \|f'\|_\infty \leq 1, \; \|f''\|_\infty \leq 2 \right\}.$$
	The sup on $\mathcal F_W^0$ in \ref{eq:Stein} coincides with a sup on $\mathcal F_W$ if $F$ is centered. To see that, it is enough to replace $f \in \mathcal F_W$ by $f-f(0)$ and use the fact that $\E[f(0)F]=0$. 
\end{remark}

\section{Main results}

We present in Section \ref{section:main} a general bound on the Wasserstein distance between a given Hawkes functional and a centered Gaussian distribution. This bound is then applied in Section \ref{section:BE} to provide Berry-Ess\'een bounds for CLTs for compound Hawkes processes. We refer to Section \ref{section:review} for a review of our results on quantitative counterparts of limit theorems (\ref{eq:introconv2})-(\ref{eq:introconv3}) regarding the compound Hawkes process.

\subsection{The general result}
\label{section:main}

Recall that we consider the compound Hawkes process given as the unique solution to SDE (\ref{eq:H}) 
\begin{equation*}
\left\lbrace
\begin{array}{l}
X_t = \int_{(0,t]\times \real_+\times \real} x \textbf{1}_{\{\theta \leq \lambda_s\}} N(ds,d\theta,dx),\quad t \geq 0 \\\\
H_t = \int_{(0,t]\times \real_+\times \real} \textbf{1}_{\{\theta \leq \lambda_s\}} N(ds,d\theta,dx),\quad t \geq 0 \\\\
\lambda_t = \mu + \int_{(0,t)} \Phi(t-u) d H_u, \quad t \geq 0.\\
\end{array}
\right.
\end{equation*}

\hspace{1em}

\noindent
\begin{assumption}
\label{assumptionY1}
Throughout this paper we recall that $\nu$ is a probability measure on $(\real,\mathcal B(\real))$ with $\nu(\{0\})=0$. In addition, we assume that : 
\begin{equation}
\label{eq:nu}
m:=\displaystyle{\int_\real x \nu(dx)}< +\infty, \quad \displaystyle{\vartheta^2 :=\int_\real x^2 \nu(dx)}< +\infty.
\end{equation}
\end{assumption}

\begin{assumption}
\label{assumptionPhi1}
The mapping $\Phi$ is such that 
$$ \|\Phi\|_1=\int_0^{+\infty} \Phi(u) du < 1.$$
This assumption is fundamental for our analysis as it allows us to set 
\begin{equation}
\label{eq:Psi}
\psi:=\sum_{n\geq 1} \phi ^{(*n)},
\end{equation} 
with $\phi^{(*n)}$ the $n$-th convolution of $\phi$ with itself.
We have that 
$$  \int_0^{+\infty} \psi (t) dt =\int_0^{+\infty}\sum_{n\geq 1} \phi ^{(*n)}(t)dt
    =\sum_{n\geq 1}\int_0^{+\infty}\phi ^{(*n)}(t)dt
    =\sum_{n\geq 1} \|\phi\|_1^n
    =\frac{\|\phi\|_1}{1-\|\phi\|_1}.$$

\end{assumption}

\begin{assumption}
\label{assumptionPhi2}
The mapping $\Phi$ is such that 
$$ \int_0^{+\infty} u \Phi(u) du < +\infty.$$
\end{assumption}

\begin{theorem}
\label{th:main}
Consider 
\begin{itemize}
\item[(i)] $Z:=(Z_{(t,x)})_{(t,x) \in \real_+\times \real}$ a $\mathbb F^X$-predictable stochastic process such that\\
\hspace*{1.3cm}$\quad \quad \quad \E\left[\int_{\real_+\times \real} |Z_{(t,x)}|^2 \lambda_t dt \nu(dx)+ \left(\int_{\real_+\times \real} Z_{(t,x)} \lambda_t dt \nu(dx)\right)^2\right]<\infty.$
\item[(ii)] $\mathcal Z:=(\mathcal Z_{(t,\theta)})_{(t,\theta) \in \real_+^2}$ the stochastic process defined as  
$$ \mathcal Z_{(t,\theta)}= \textbf{1}_{\{\theta \leq \lambda_t\}}, \quad (t,\theta) \in \real_+^2.$$
\item[(iii)] $F:=\delta^N(Z \mathcal Z) = \int_{\real_+^2\times \real} Z_{(t,x)} \mathcal Z_{(t,\theta)} N(dt,d\theta,dx) - \int_{\real_+\times \real} Z_{(t,x)} \lambda_t dt \nu(dx).$
\end{itemize}
Let $\gamma>0$.
Then, letting $G \sim \mathcal N(0,\gamma^2)$, 
\begin{equation}
\label{eq:mainestimate}
d_W(F,G) \leq \E\left[\left|\gamma^2 - \int_{\real_+\times \real} Z_{(t,x)} \lambda_t D_{(t,\lambda_t,x)} F dt \nu(dx)\right| \right] + \E\left[\int_{\real_+\times \real} |Z_{(t,x)}| \lambda_t \left| D_{(t,\lambda_t,x)} F \right|^2 dt \nu(dx)\right].
\end{equation}
\end{theorem}

\begin{proof}
Note first of all that by Theorem \ref{th:IPPH}, $F$ is well-defined.
Then, let $f$ in $\mathcal{F}_W^0$ (see Definition \ref{def:W}). As $F = \delta^N(Z \mathcal Z)$, the integration by parts formula (\ref{eq:IBP}) leads to  
\begin{align*}
\E\left[f(F) F\right] & = \E\left[f(F) \delta^N(Z \mathcal Z)\right]\\
&= \E\left[\int_{\real_+\times \real} Z_{(t,x)} \lambda_t D_{(t,\lambda_t,x)} f(F) dt \nu(dx)\right] \\
&= \E\left[\int_{\real_+\times \real} Z_{(t,x)} \lambda_t \left(f(F \circ \eps_{(t,\lambda_t,x)}^+)-f(F)\right) dt \nu(dx)\right].
\end{align*}
Using a Taylor expansion, (with $\bar{F}$ be a random element between $F_T \circ \eps_{(t,\lambda_t,x)}^+$ and $F$) we have 
$$ f(F \circ \eps_{(t,\lambda_t,x)}^+)-f(F) = f'(F) D_{(t,\lambda_t,x)} F + \frac12 f''(\bar{F}) \left| D_{(t,\lambda_t,x)} F \right|^2.$$
Hence, 
\begin{align*}
&\left|\E\left[\gamma^2 f'(F) - f(F) F\right]\right|\\
&= \left|\E\left[f'(F) \left(\gamma^2 - \int_{\real_+\times \real} Z_{(t,x)} \lambda_t D_{(t,\lambda_t,x)} F dt \nu(dx)\right) \right] - \frac{1}{2} \E\left[\int_{\real_+\times \real} Z_{(t,x)} \lambda_t f''(\bar{F}) \left| D_{(t,\lambda_t,x)} F \right|^2 dt \nu(dx)\right]\right|\\
&\leq \|f'\|_\infty \E\left[\left|\gamma^2 - \int_{\real_+\times \real} Z_{(t,x)} \lambda_t D_{(t,\lambda_t,x)} F dt \nu(dx)\right| \right] + \frac{\|f''\|_\infty}{2} \E\left[\int_{\real_+\times \real} |Z_{(t,x)}| \lambda_t \left| D_{(t,\lambda_t,x)} F \right|^2 dt \nu(dx)\right]\\ 
&\leq \E\left[\left|\gamma^2 - \int_{\real_+\times \real} Z_{(t,x)} \lambda_t D_{(t,\lambda_t,x)} F dt\right| \right] + \E\left[\int_{\real_+\times \real} |Z_{(t,x)}| \lambda_t \left| D_{(t,\lambda_t,x)} F \right|^2 dt \nu(dx)\right],
\end{align*}
as $f$ belongs to $\mathcal{F}_W^0$. 
\end{proof}
\noindent In particular, we obtain the following corollary.
\begin{corollary}
\label{co:main}
For each $T>0$, consider $Z^T:=(Z_{(t,x)}^T)_{(t,x) \in \real_+\times \real}$ defined as : 
$$ Z_{(t,x)}^T:=x \alpha_t \textbf{1}_{t \in [0,T]}, \quad t \geq 0, \; x \in \real,$$
where $(\alpha_t)_{t\in [0,T]}$ is a $(\mathcal F_t^X)_{t\in [0,T]}$-predictable  process. Let also : 
\begin{align*}
F_T:=\delta^N(Z^T \mathcal Z) &= \int_{(0,T]\times \real_+\times \real} \alpha_s x \textbf{1}_{\{\theta \leq \lambda_s\}} (N(ds,d\theta,dx)-ds d \theta \nu(dx)) \\
&= \int_{(0,T]} \alpha_t dX_t - m \int_0^T \alpha_t \lambda_t dt, \quad m:=\int_\real x \nu(dx);
\end{align*}
and $G \sim \mathcal N(0,\gamma^2)$ for  any $\gamma$. Then 
\begin{equation}
\label{eq:mainestimatebis}
d_W(F_T,G) \leq \E\left[\left|\gamma^2 - \int_0^T \alpha_t \lambda_t \int_\real x D_{(t,\lambda_t,x)} F_T \nu(dx) dt\right| \right] + \E\left[\int_0^T |\alpha_t| \lambda_t \int_\real |x| \left| D_{(t,\lambda_t,x)} F_T \right|^2 \nu(dx) dt\right].
\end{equation}
\end{corollary}
\noindent Before going further we make the following remark showing that somehow the decomposition above is sharp. 

\begin{remark}
\label{rk:Poisson1}
Assume $\Phi \equiv 0$ so that $H$ is an homogeneous Poisson process with intensity $\mu$. Let $m=\int_\real y \nu(dy)$ and $X$ is a compound Poisson process which can thus be represented as : 
$$ X_t = \sum_{i=1}^{H_T} Y_i,$$
with $(Y_i)_{i\geq 1}$ iid random variables independent of $H$ with distribution $\nu$. Assume in addition that  $\displaystyle{\int_{\real} |x|^3 \nu(dx) <+\infty}$. We have that 
$$ F_T = \frac{X_T - \mu m T}{\sqrt{T}},$$
$\lambda_t = \mu$ and 
$$ D_{(t,\lambda_t,x)} F_T = \frac{X_T + x - X_T}{\sqrt{T}} =  \frac{x}{\sqrt{T}}.$$
In addition, 
$${\gamma^2 = \mu \vartheta^2= \mu \int_\real y^2 \nu(dy).} $$
As a consequence,  {for $Z_{(t,x)}^T= \frac{x}{\sqrt{T}}$ for all $t \in [0,T]$ in the previous corollary, we get}
$$ \E\left[\left|\gamma^2 - \frac{1}{\sqrt{T}} \int_0^T \int_\real \lambda_t x D_{(t,\lambda_t,x)} F_T dt \nu(dx)\right| \right] = \left|\gamma^2 - \mu  \vartheta^2 \right| =0.$$
Hence, the speed of convergence is completely contained in that case by the term  
$$\frac{1}{\sqrt{T}} \E\left[\int_0^T \int_\real \lambda_t  |x| \left| D_{(t,\lambda_t,x)} F_T \right|^2 dt \nu(dx)\right]\\
=\frac{\mu \int_\real |x|^3 \nu(dx)}{\sqrt{T}}$$  
and thus we recover a Berry-Ess\'een bound with central speed $T^{-1/2}$.
\end{remark}

\begin{remark}
\label{rem:NPA}
Before concluding this section we would like to make a comment regarding the Nourdin-Peccati methodology we applied with a slight modification. Indeed, one realises that the key ingredient is to consider a random variable $F$ of the form $F=\delta(u)$ where $u$ is a given process (belonging to an appropriate class). In the Nourdin-Peccati's approach (in both Gaussian and Poisson frameworks), one consider a centered random variable $F$ that then naturally belongs to the domain of the Ornstein-Uhlenbeck's operator $L$, which can be defined as $L F = -\delta(D F)$, where once again in both Gaussian and Poisson frameworks $D$ is the Malliavin derivative and $\delta$ the divergence operator. Hence, writing $F = L L^{-1} F = - \delta(D L^{-1} F)$ one gets back to the previous divergence form.\\\\
\noindent
Coming back to the notations of Theorem \ref{th:main} (and choosing $\nu(dx)=\delta_{1}(dx)$ for simplicity of notations), for a Hawkes functional $F$ of the form $F=\delta^N(Z \mathcal Z)$, by adapting the Nourdin-Peccati methodology to our framework (so by adapting \cite{Peccati_2010}), one would obtain the inequality (once again using the notations of Theorem \ref{th:main}): 
\begin{equation}
\label{eq:mainestimateNP}
d_W(F,G) \leq \E\left[\left|\gamma^2 - \int_0^{+\infty} \lambda_t \left(-D_{(t,\lambda_t)} L^{-1} F\right) D_{(t,\lambda_t)} F dt\right| \right] + \E\left[\int_0^{+\infty} |D_{(t,\lambda_t)} L^{-1} F| \lambda_t \left| D_{(t,\lambda_t)} F \right|^2 dt\right].
\end{equation}
However, as $F$ is of the form $F=\delta^N(Z \mathcal Z)$, bounds (\ref{eq:mainestimate}) and (\ref{eq:mainestimateNP}) coincide as : 
$$ -D_{(t,\lambda_t)} L^{-1} F = Z_t, \quad \P\otimes \lambda_t dt-\textrm{a.e.},$$
or equivalently 
$$ -D_{(t,\theta)} L^{-1} F = Z_t \mathcal Z_{(t,\theta)}, \quad \P\otimes dt\otimes d\theta-\textrm{a.e.}.$$
Indeed, letting $\rho_{(t,\theta)} := -D_{(t,\theta)} L^{-1} F$, it holds that
$$\delta^N(\rho) =  -\delta^N(-\rho)=-\delta^N(D L^{-1} F) =L L^{-1} F = F = \delta^N(Z \mathcal Z).$$
Finally note that the fact that the bound rewrites in a simpler form for divergence form functionals has already been observed and exploited (see for instance \cite{Privault:2018aa}, \cite{Besancon} or \cite{torrisi}).
\end{remark}

\begin{remark}
\label{rk:Torrisi}
We would like to discuss here the work \cite{torrisi} which is the closest to ours. In \cite{torrisi}, the author makes use of the Poisson embedding representation to apply the Nourdin-Peccati methodology. So on that regard our work follows the same point of view. However, the remark (initially pointed out in \cite{Hillairet_Reveillac_Rosenbaum} and fully exploited in this paper) that for a Hawkes process the impact of the operator $D_{(t,\theta,x)}$ on a Hawkes functional is the same regarding the value of $\theta$ (provided that $\theta \leq \lambda_t$) was not found. This remark is made clear in Proposition \ref{prop:TempDH} and leads to a simplification of the integration by parts formula (see Relation (\ref{eq:IBP})) for Hawkes functional and thus of the bound in Theorem \ref{th:main2} below. In \cite{torrisi}, the general case of a stochastic intensity point process is considered (including non-linear Hawkes processes), but then calls for some estimates that turn out to fail to be sharp enough in the particular case of a linear Hawkes process. For instance, when dealing with the particular case of a Hawkes process, plugging $u(t):=\frac{1}{\sqrt{T}} \textbf{1}_{[0,T}(t)$ in \cite[Theorem 4.1]{torrisi} leads to an upper bound which does not converges to $0$ as $T$ goes to $+\infty$ (we chose the linear case with $\Phi(x):=\mu+x$ with the notations of \cite[Theorem 4.1]{torrisi}). This might have been a motivation to consider an approximation of the Hawkes process by a counting process ($\delta_a$ in \cite[p.~2120]{torrisi}) which is interesting but which does not cover strictly speaking Convergence (\ref{eq:introconv3}).  
\end{remark}

\subsection{Application to the Hawkes process}
\label{section:BE}

\subsubsection{Overview of our results}
\label{section:review}

We aim to provide the speed of convergence for the convergence of the renormalized Hawkes process when $T$ tends to $+\infty$.\\\\
\noindent
We recall that in the case of a (non-compound) Hawkes process (that is $X_t = H_t$, $\nu(dx)=\delta_{1}(dx)$ or equivalently $Y_i \equiv 1$ in (\ref{eq:compound})), it has been proved as \cite[Lemma 7]{Bacry_et_al_2013} that 
\begin{equation*}
\left(\frac{H_{Tv} - \int_0^{Tv} \lambda_s ds}{\sqrt{T}}\right)_{v \in [0,1]} \underset{T\to+\infty}{\overset{\mathcal{L}-\mathcal{S}}{\longrightarrow}} ({\sigma} W_v)_{v\in [0,1]},
\end{equation*}
(with $W$ a Brownian motion and ${\sigma}^2:=\frac{\mu}{1-\|\Phi\|_1}$)  which in particular implies 
\begin{equation}
\label{eq:convBacrietal}
\frac{H_{T} - \int_0^{T} \lambda_s ds}{\sqrt{T}} \underset{T\to+\infty}{\overset{\mathcal{L}}{\longrightarrow}} \mathcal{N}(0,{\sigma}^2).
\end{equation}
In what follows, we provide for a compound Hawkes process $X:=(X_t)_{t\in [0,T]}$, 
$$ X_t = \sum_{i=1}^{H_t} Y_i,$$
($(Y_i)_{i\geq 1}$ iid with common distribution $\nu$ and independent of $H$) the counterpart of the normal convergence (\ref{eq:convBacrietal}). Let 
$$ F_T:= \frac{X_T - m \int_0^T \lambda_t dt}{\sqrt{T}}, \quad \sigma^2:=\frac{\mu}{1-\|\Phi\|_1}; \quad  \vartheta^2=\int_\real y^2 \nu(dy).$$
More precisely 
\begin{enumerate}
\item We give, as Theorem \ref{th:main2}, a general bound  for the speed of convergence (with respect to $T$) of the convergence (\ref{eq:convBacrietal}) in the generalized case of a compound Hawkes process. 
\item We prove, as Theorem \ref{th:main3} (see this result for a precise statement), for  a compound Hawkes process, and in case of an exponential kernel $\Phi(u)=\alpha e^{-\beta u}$, $u\geq 0$ (with $0 < \alpha < \beta$) or an Erlang kernel $\Phi(u)=\alpha u e^{-\beta u}$, $u\geq 0$  (with $0 < \alpha < \beta^2$) the speed of convergence $O\left(\frac{1}{\sqrt{T}}\right)$ for the counterpart of convergence (\ref{eq:convBacrietal}) :
$$ d_W\left(F_T,\mathcal{N}\left(0,\sigma^2 \ \vartheta^2\right)\right) \leq \frac{C_{\alpha,\beta,\nu}}{\sqrt{T}}; $$
\item We provide, as Theorem \ref{th:main4}, for the Hawkes process and in case of an exponential kernel or of an Erlang kernel, a speed of convergence (with respect to the Wasserstein distance) of the modified CLT in the spirit of \cite[Theorem 2]{Bacry_et_al_2013} as follows : 
$$d_W\left(Y_T,\mathcal{N}(0,\tilde{\sigma}^2)\right) \leq \frac{\tilde{C}_{\alpha,\beta}}{\sqrt{T}},$$
where $Y_T:=\frac{H_T-\int_0^T \E[\lambda_t]dt}{\sqrt T}$ and $\tilde {\sigma}^2=\frac{\mu}{(1-\|\Phi\|_1)^3}$.
\end{enumerate}

\subsubsection{Quantitative Limit Theorems for compound Hawkes processes}

In order to make precise some of the statements below we recall that the Malliavin derivative of $X$, $H$ and $\lambda$ involves the following parametrized system (see Proposition \ref{prop:DescDH}). 

\begin{notation}
\label{notation:systemhat}
For fixed $t\geq 0$, we denote by $(\hat X_s^t,\hat H_s^t,\hat \lambda_s^t)_{s\geq t}$ the unique solution to the SDE 
\begin{equation*}
\left\lbrace
\begin{array}{l}
\hat H_s^{t} = \displaystyle{\int_{(t,s]\times \real_+\times \real} \textbf{1}_{\{\lambda_u \leq \theta \leq \lambda_u+ \hat \lambda_u^{t}\}} N(du,d\theta,dy)}, \quad s \geq t, \\\\
\hat X_s^{t} = \displaystyle{\int_{(t,s]\times \real_+\times \real} y \textbf{1}_{\{\lambda_u \leq \theta \leq \lambda_u+ \hat \lambda_u^{t}\}} N(du,d\theta,dy)}, \quad s \geq t,\\\\
\hat \lambda_s^{t} =  \Phi(s-t) + \displaystyle{\int_{(t,s)} \Phi(s-u) d\hat H_u^{t}}, \quad s>t, \; \hat \lambda_t^{t}=0.
\end{array}
\right.
\end{equation*}
In addition we introduce different compensated martingale processes for the (shifted) Hawkes process and the (shifted) compound Hawkes process
$$ M_s:= X_s- m \int_0^s \lambda_u du, \quad s \in [0,T], \quad 
\hat M_s^{t}:= \hat X_s^{t}- m \int_t^s \hat \lambda^t_u du, \quad s \in [t,T].$$
$$ \mathcal M_s:=  H_s-  \int_0^s \lambda_u du, \quad s \in [0,T], \quad 
 \hat { \mathcal M_s^{t}}:= \hat H_s^{t}-  \int_t^s \hat \lambda^t_u du, \quad s \in [t,T].$$
\end{notation}

\begin{theorem}
\label{th:main2}
Assume Assumptions \ref{assumptionY1}, \ref{assumptionPhi1} and \ref{assumptionPhi2} are in force. 
For $T>0$, let $Z^T_{(t,x)}:= \frac{x}{\sqrt{T}}$, $t\in [0,T]$ and recall
$$F_T = \delta^N(Z^T \mathcal Z^T) = \frac{X_T - m \int_0^T \lambda_t dt}{\sqrt{T}}.$$ 
Let also  $G\sim \mathcal N(0,\sigma^2 \ \vartheta^2)$ with $\sigma^2:=\frac{\mu}{1-\|\Phi\|_1}$ (recall $\vartheta^2=\int_\real y^2 \nu(dy)$). There exists a constant $C_{\Phi,\nu} >0$ depending on $\Phi$ and $\nu$ only such that for any $T>0$
$$ d_W(F_T,G) \leq \frac{C_{\Phi,\nu}}{\sqrt{T}} + \frac{|m|}{T} \E\left[\left| \int_0^T \lambda_t \hat M_T^{t} dt \right| \right],$$
where $\hat M^t$ is defined in Notation \ref{notation:systemhat}.
\end{theorem}

\begin{proof}
To ease the presentation, we chose to provide here the main steps of the proof and to postpone in Section  \ref{section:teclemma} the proof of general results concerning the convergence of moments of order 2.
By Theorem \ref{th:main} we have that 
$$d_W(F_T,G) \leq A_1 +A_2$$
where 
$$ A_1:= \E\left[\left|\sigma^2 \, \vartheta^2- \frac{1}{\sqrt{T}} \int_0^T\int_\real x \lambda_t D_{(t,\lambda_t,x)} F_T dt \nu(dx)\right| \right], \quad A_2:=\frac{1}{\sqrt{T}} \E\left[\int_0^T \int_\real \lambda_t |x| \left| D_{(t,\lambda_t,x)} F_T \right|^2 dt \nu(dx)\right].$$ 
Note that 
$$\sigma^2 \, \vartheta^2= \lim_{T\to +\infty} \E[|F_T|^2] = \lim_{T\to +\infty} \frac{1}{T}\E[|M_T|^2] = d \lim_{T\to +\infty} \frac{1}{T}\E[H_T] = d \lim_{T\to +\infty} \frac{1}{T} \int_0^T \E[\lambda_t] dt,$$
and we will quantify the speed of  convergence. Before going further, recall that according to Proposition \ref{prop:DescDH}, 
\begin{equation}
\label{eq:decompositionD}
D_{(t,\lambda_t,x)} F_T = \frac{1}{\sqrt{T}} D_{(t,\lambda_t,x)} M_T = \frac{1}{\sqrt{T}} \left(x + \hat M_T^{t}\right)
\end{equation}
where $(\hat M_s^{t}:= \hat X_s^{t}- m \int_t^s \hat \lambda^t_u du)_{s \in [t,T]}$, and $(\hat H^{t},\hat \lambda^{t})$ is defined  by (\ref{eq:DH}) and recalled in Notation \ref{notation:systemhat}.
We treat both terms $A_1$ and $A_2$ separately.\\\\
\textbf{Term $A_1$ \\}
We have 
\begin{align*}
A_1 &= \E\left[\left|\sigma^2 \, \vartheta^2 - \frac{1}{\sqrt{T}} \int_0^T \int_\real x \lambda_t D_{(t,\lambda_t,x)} F_T dt \nu(dx)\right| \right]\\
&=\E\left[\left|\sigma^2 \, \vartheta^2 - \frac{\vartheta^2}{T} \int_0^T \lambda_t dt - \frac{1}{T} \int_0^T \int_\real x \lambda_t \hat M_T^{t} dt \nu(dx)\right| \right] \\
&\leq \left|\sigma^2 \, \vartheta^2 - \frac{\vartheta^2}{T} \int_0^T \E[\lambda_t] dt \right| +\frac{\vartheta^2}{T} \E\left[\left| \int_0^T (\lambda_t-\E[\lambda_t]) dt\right|\right]+ \frac{|m|}{T} \E\left[\left| \int_0^T \lambda_t \hat M_T^{t} dt \right| \right]\\
&=: A_{1,1} + \vartheta^2 A_{1,2} + |m| A_{1,3}.
\end{align*}
By Lemma \ref{lemma:cvgsigma} (in Section \ref{section:teclemma}), 
\begin{equation}
\label{eq:A11speed}
A_{1,1} = O\left(\frac{1}{T}\right).
\end{equation}
We turn to term $A_{1,2}$. 
According to the second line of \cite[Lemma 4]{Bacry_et_al_2013} one has
$$H_T-\mathbb E[H_T]=\mathcal M_T+\int_0^T \psi(T-s)\mathcal M_sds,$$
where $\mathcal M_t=H_t-\int_0^t \lambda_s ds.$
By subtracting $H_T$ from the equation one gets
$$-\mathbb E[H_T]=-\int_0^T \lambda_s ds+\int_0^T\psi(T-s)\mathcal M_sds,$$
thus since $\mathbb E[H_T]=\int_0^T\mathbb E [\lambda_s] ds$ 
$$\int_0^T \lambda_s-\mathbb E [\lambda_s] ds=\int_0^T\psi(T-s)\mathcal M_sds.$$
This means that
$$A_{1,2}\leq \frac{1}{T}\int_0^T\psi(T-s)\mathbb E[|\mathcal M_s|]ds
\leq \frac{1}{T}\int_0^T\psi(T-s)\mathbb E[|\mathcal M_s|^2]^{1/2}ds
\leq \frac{1}{T}\int_0^T\psi(T-s)\mathbb E[H_s]^{1/2}ds.$$
According to the proof of \cite[Lemma 5]{Bacry_et_al_2013}, with  Assumption \ref{assumptionPhi2} that $\int_0^{+\infty}s \phi(s)ds<+\infty$,   then\\
 $C:=\int_0^{+\infty}s \psi(s)ds<+\infty$ and a fortiori $\forall p \in [0,1] ,\int_0^{+\infty}s^p \psi(s)ds<+\infty.$ Following once again \cite[Lemma 4]{Bacry_et_al_2013},
$$   \mathbb E[H_t]=\mu t +\int _0^t \psi(t-s) s ds
    =\mu t +\int_0^t \psi(s)(t-s)ds
    \leq (\mu+\|\psi\|_1)t+ C.$$
Hence, using $\sqrt{a+b}\leq \sqrt{a}+\sqrt{b}$ we have the following inequality
$$\mathbb E[H_t]^{1/2}\leq A \sqrt{t}+B.$$ 
Now the term $A_{1,2}$ becomes bounded by
\begin{align*}
A_{1,2} &\leq \frac{A}{T}\int_0^T\psi(T-s)\sqrt{s}ds+\frac{B}{T}\int_0^T\psi(T-s)ds,\\
&\leq \frac{A}{T}\int_0^T\psi(s)\sqrt{T-s}ds+\frac{B}{T}\|\psi\|_1,\\
&\leq \frac{A}{\sqrt T}\int_0^T\psi(s)\sqrt{\frac{T-s}{T}}ds+\frac{B}{T}\|\psi\|_1,\\
&\leq \frac{A}{\sqrt T}\int_0^T\psi(s)ds+\frac{B}{T}\|\psi\|_1,\\
&\leq (\frac{A}{\sqrt T}+\frac{B}{T})\|\psi\|_1,
\end{align*}
leading to
\begin{equation}
\label{eq:A12speed}
A_{1,2} = O\left(\frac{1}{\sqrt{T}}\right).
\end{equation}
Combining (\ref{eq:A11speed}) and (\ref{eq:A12speed}) we get that 
$$ A_1 = O\left(\frac{1}{\sqrt{T}}\right) + {\frac{|m|}{T} \E\left[\left|\int_0^T\lambda_t \hat M_T^t dt\right|\right]}.$$
\textbf{Term $A_2$  \\}
Recall that
\begin{align*}
A_2 &=\frac{1}{\sqrt{T}} \E\left[\int_0^T \int_\real \lambda_t |x| \left| D_{(t,\lambda_t,x)} F_T \right|^2 \nu(dx) dt\right]\\
&\leq \frac{2 \int_\real |x|^3 \nu(dx)}{T^{3/2}} \E\left[\int_0^T \lambda_t dt\right] + \frac{2}{T^{3/2}} \E\left[\int_0^T \int_\real |x| \lambda_t \left|\hat M_T^t\right|^2 \nu(dx) dt\right]\\
&= \frac{2 \int_\real |x|^3 \nu(dx)}{T^{3/2}} \E\left[\int_0^T \lambda_t dt\right] + \frac{2 \int_\real |x| \nu(dx)}{T^{3/2}} \E\left[\int_0^T \lambda_t \left|\hat M_T^t\right|^2 dt\right]\\
&=: 2 \left(\int_\real |x|^3 \nu(dx) A_{2,1} + \int_\real |x| \nu(dx) A_{2,2}\right).
\end{align*}
By Lemma \ref{lemma:cvgsigma}, we immediately get that 
\begin{equation}
\label{eq:A21speed} 
A_{2,1} = O\left(\frac{1}{\sqrt{T}}\right).
\end{equation}
Finally, for Term $A_{2,2}$, it holds that 
\begin{align*}
A_{2,2} &= \frac{1}{T^{3/2}} \E\left[\int_0^T \lambda_t \left|\hat M_T^t\right|^2 dt\right]\\
&= \frac{1}{T^{3/2}} \int_0^T \E\left[ \lambda_t \E\left[\left|\hat M_T^t\right|^2\vert \mathcal F_t\right]\right] dt \\
&= \frac{1}{T^{3/2}} \int_0^T \E\left[ \lambda_t \E\left[\int_t^T \int_{\real} \int_{\real_+} x^2 \textbf{1}_{\{\theta \leq \hat \lambda_u^t\}} N(du,d\theta,dx)\vert \mathcal F_t\right]\right] dt \\
&= \frac{\vartheta^2}{T^{3/2}} \int_0^T \E\left[ \lambda_t \int_t^T \E\left[\hat \lambda_s^t\vert \mathcal F_t\right] ds \right] dt
\end{align*}
{where we used  for the last equality the identity  $\hat H_t^t =0$.} Using Lemma \ref{lemma:momentDH}, 
$$|A_{2,2}|\leq \vartheta^2 \frac{\|\phi\|_1(1+ \|\psi\|_1)}{T^{3/2}}\int_0^T \mathbb E[\lambda_t]dt,$$
and since $\int_0^T \mathbb E[\lambda_t]dt =\mathbb E[H_T] \leq AT+B$ where $A$ and $B$ are positive constants (see computations for Term $A_{1,1}$ above), we have
\begin{equation}
\label{eq:A22speed} 
A_{2,1} = O\left(\frac{1}{\sqrt{T}}\right).
\end{equation}
Combining (\ref{eq:A21speed}) and (\ref{eq:A22speed}), we get that 
$$ A_{2} = O\left(\frac{1}{\sqrt{T}}\right). $$
\end{proof}

\begin{remark}
For simplicity let $R_T:=\frac{1}{T} \E\left[\left|\int_0^T \lambda_t  \hat M^t_T dt\right|\right]$. 
 According to Theorem \ref{th:main2},  determining the speed of convergence for a general Hawkes process requires the speed of convergence for $R_T$. It is somehow embarrassing to admit that we were not able to deal with it for a general kernel $\Phi$ as it calls for a precise statement of the correlation between the original intensity $\lambda$ and somehow the one of the shifted (more precisely of the Malliavin derivative) Hawkes martingale  $\hat M$ (defined in  Notation \ref{notation:systemhat}). It is worth noticing that getting estimates on the correlation of the Hawkes process itself is already quite challenging for a general kernel and constitutes an active research area (see \textit{e.g.} \cite{Horst_Wei_2019}). However, in particular cases we can provide estimates on this quantity $R_T$ as we will see in Theorem \ref{th:main3}. It is also interesting to point out that this term is specific to the self-exciting feature of the intensity as in the Poisson case, that is when $\Phi \equiv 0$, we have that $R_T=0$ for any $T$ (see Remark \ref{rk:Poisson1}).
\end{remark}

\begin{theorem}
\label{th:main3}
Assume an exponential kernel $\Phi(u)=\alpha e^{-\beta u}$, $u\geq 0$ (with $0 < \alpha < \beta$) or an Erlang kernel $\Phi(u)=\alpha u e^{-\beta u}$, $u\geq 0$ (with $0 < \alpha < \beta^2$).\\ Then, using the notations of Theorem \ref{th:main2} ($F_T = \frac{H_T - \int_0^T \lambda_t dt}{\sqrt{T}}$; $G \sim \mathcal N(0,\sigma^2 \, \vartheta^2)$$)$, there exists  in both cases a constant $C_{\alpha,\beta,\nu} >0$ depending only on $\alpha,\beta$ and $\nu$ such that for any $T>0$
$$ d_W(F_T,G) \leq \frac{C_{\alpha,\beta,\nu}}{\sqrt{T}}.$$
\end{theorem}

\begin{proof}
Obviously, Assumptions (\ref{assumptionPhi1})-(\ref{assumptionPhi2}) are in force. By Theorem \ref{th:main2}, we need to estimate the quantity 
$$R_T=\frac{1}{T} \E\left[\left|\int_0^T \lambda_t \hat M_T^t dt\right|\right]=\frac{1}{T} \E\left[\left|\int_0^T \lambda_t \left(D_{(t,\lambda_t)} M_T - 1\right) dt\right|\right],$$
where $\hat M^t$ is defined in Notation \ref{notation:systemhat}.
In Lemma \ref{lemma:A13} we prove that $R_T =O(\frac{1}{\sqrt{T}}).$
\end{proof}

\subsubsection{Alternative quantitative Limit Theorem for the exponential and Erlang Hawkes processes}

In this section we consider the Hawkes process $H$ (alternatively one can set $\nu(dx)=\delta_1(dx)$ or $Y_i \equiv 1$ in Representation (\ref{eq:compound})). 
It has been proven in \cite{Bacry_et_al_2013} that as $T$ goes to infinity
$$Y_T\underset{T\to+\infty}{\overset{\mathcal{L}}{\longrightarrow}} \mathcal{N}(0,\tilde{\sigma}^2),$$
where $Y_T=\frac{H_T-\int_0^T \E[\lambda_t]dt}{\sqrt T}$ is a centered and normalized Hawkes process and $\tilde {\sigma}^2=\frac{\mu}{(1-\|\Phi\|_1)^3}.$\\
The goal of this section is to provide the speed of convergence of $Y_T$ using Wasserstein metric between $F_T$ and its Gaussian limit that we have established in the last paragraph.
\begin{theorem}
\label{th:main4}
	Set $Y_T=\frac{H_T-\int_0^T \E[\lambda_t]dt}{\sqrt T}$ and $\tilde {\sigma}^2=\frac{\mu}{(1-\|\Phi\|_1)^3}$.\\
	Assume { an exponential kernel }$\Phi(u)=\alpha e^{-\beta u}$, $u\geq 0$ (with $0 < \alpha < \beta$) or  { an Erlang kernel } $\Phi(u)=\alpha u e^{-\beta u}$, $u\geq 0$ (with $0 < \alpha < \beta^2$).\\ Then, if  $G \sim \mathcal N(0,\tilde \sigma^2)$ there exists in both cases a constant $C_{\alpha,\beta} >0$ depending only on $\alpha,\beta$ such that for any $T>0$
	$$ d_W(Y_T,G) \leq \frac{C_{\alpha,\beta}}{\sqrt{T}}.$$
\end{theorem}
\begin{proof}
	As we have shown in Lemma \ref{lemma:alt}, it is possible to link $Y_T$ to $F_T$ via the relation $$\frac{Y_T}{\gamma}=F_T+\mathfrak R_T$$ where $\gamma=\frac{1}{1-\|\Phi\|_1}$ is a positive constant and  $\mathfrak R_T$ is a "small remainder" whose expressions depend on the kernel. According to $\eqref{eq:Stein}$
	$$d_W\left(\frac {Y_T}{\gamma},\frac{G}{\gamma}\right) \leq \sup _{f\in \mathcal F_W^0} \left|\E\left[\frac{\tilde \sigma ^2}{\gamma^2}f'\left(\frac{Y_T}{\gamma}\right)-\frac{Y_T}{\gamma} f\left(\frac{Y_T}{\gamma}\right)\right]\right|,$$
	where $G$ is a Gaussian of variance $\tilde \sigma ^2 =\frac{\mu}{(1-\|\phi\|_1)^3}.$\\
	By applying a Taylor expansion on $f$ and $f'$:
	$$f\left(\frac{Y_T}{\gamma}\right)=f(F_T+\mathfrak  R_T)=f(F_T)+\mathfrak  R_T f'(X^*),$$
	$$f'\left(\frac{Y_T}{\gamma}\right)= f'(F_T+\mathfrak  R_T)= f'(F_T)+\mathfrak  R_T f''(\bar X),$$
	with $X^*$, $\bar X$ in $[F_T \wedge (F_T+\mathfrak  R_T), F_T \vee (F_T+\mathfrak  R_T)]$ two random variables. Thus:
	\begin{align*}
	&d_W\left(\frac {Y_T}{\gamma},\frac{G}{\gamma}\right) \\
	\leq& \sup _{f\in \mathcal F_W^0} \left|\E\left[\frac{\tilde \sigma^2}{\gamma^2}f'(F_T+\mathfrak  R_T) - (F_T+\mathfrak  R_T) f(F_T+\mathfrak  R_T)\right]\right|,\\
	=& \sup _{f\in \mathcal F_W^0} \left|\E\left[\frac{\tilde \sigma^2}{\gamma^2} (f'(F_T)+\mathfrak  R_T f''(\bar X))-(F_T+\mathfrak  R_T)(f(F_T)+\mathfrak  R_T f'(X^*))\right]\right|,\\
	=& \sup _{f\in \mathcal F_W^0} \left|\E\left[\frac{\tilde \sigma^2}{\gamma^2} f'(F_T) -F_T f(F_T)+\frac{\tilde \sigma^2}{\gamma^2} \mathfrak  R_T f''(\bar X)-\mathfrak  R_T f(F_T) -F_T \mathfrak  R_T f'(X^*) -\mathfrak  R_T^2 f'(X^*))\right]\right|,\\
	\leq& \sup _{f\in \mathcal F_W^0} \left|\E\left[\frac{\tilde \sigma^2}{\gamma^2} f'(F_T)-F_T f(F_T)\right]\right|\\
	+&\sup _{f\in \mathcal F_W^0} \E\left[\left|\frac{\tilde \sigma^2}{\gamma^2} R_T f''(\bar X)\right|+\left|\mathfrak  R_T f(F_T)\right|+\left|F_T  \mathfrak  R_T f'(X^*)\right|+\left|\mathfrak  R_T^2 f'(X^*)\right|\right].
	\end{align*}
	The choice of the variance yields $\frac{\tilde \sigma^2}{\gamma^2} \frac{\mu}{(1-\|\Phi\|_1)^3}\cdot (1-\|\Phi\|_1)^2=\sigma^2$. In Theorem $\ref{th:main3}$ we have proven that $$ \sup _{f\in \mathcal F_W^0} \left|\E\left[\sigma^2 f'(F_T)-F_T f(F_T)\right]\right| \leq O\left(\frac{1}{\sqrt T}\right).$$
	And since $f$ is in $\mathcal F_W^0$, the first and second derivatives are bounded and the Wasserstein distance is thus bounded by
	$$d_W\left(\frac {Y_T}{\gamma},\frac{G}{\gamma}\right) \leq C \left(\frac{1}{\sqrt T}+\E [|\mathfrak  R_T|]+\sup _{f\in \mathcal F_W^0} \E [|\mathfrak  R_T f(F_T)|]+\E [|\mathfrak  R_T F_T|]+\E [|\mathfrak  R_T^2|]\right)$$
	for a positive constant $C$ that does not depend on $T$. In the computations below $C$ will denote a constant independent of $T$ which may change from line to line.\\
	We now try to simplify the upper bound. An application of Cauchy-Schwarz inequality yields
	\begin{align*}
	\mathbb E [|\mathfrak  R_T|] &\leq \sqrt {\mathbb E[\mathfrak  R_T^2]}, \quad \mathbb E [|\mathfrak  R_T F_T|] \leq \sqrt {\mathbb E [\mathfrak  R_T^2]\mathbb E [F_T^2]},
	\end{align*}
    and
    \begin{align*}
    \mathbb E [\mathfrak  R_T f(F_T)] ^2 &\leq \mathbb E[\mathfrak  R_T^2] \mathbb E[|f(F_T)|^2],\\
    &\leq  \mathbb E[\mathfrak  R_T^2] \mathbb E[\|f'\|_{\infty}^2 |F_T|^2], \quad \text{using mean value equality and the fact that $f\in \mathcal F _W^0$}\\
    &\leq \mathbb E[\mathfrak  R_T^2]  \mathbb E [|F_T|^2].
    \end{align*}
    In order to have an upper bound on $\E[|F_T|^2]$, we recall that $F_T=\frac{M_T}{\sqrt T}$ where $M_T=H_T-\int_0^T\lambda_t dt$	is a martingale. This means that $\E[|M_T|^2]=\E[H_T]=O(T)$ (the laste estimate can be found in $\cite{errais2010affine}$), which yields $\E[|F_T|^2] \leq C$ and
    $$\mathbb E[\mathfrak  R_T^2]  \mathbb E [|F_T|^2] \leq C \E[\mathfrak  R_T^2] $$
    and the Wasserstein distance is now bounded by 
    $$d_W\left(\frac {Y_T}{\gamma},\frac{G}{\gamma}\right) \leq C \left(\frac{1}{\sqrt T}+\sqrt{\mathbb E [\mathfrak  R_T^2]}\right).$$
    To perform an estimate on $\mathbb E [\mathfrak R_T^2]$ we need to distinguish the case of Exponential and Erlang kernel.\\\\
    \textbf{Case 1 : the kernel is an exponential function}\\
    In this case the remainder term writes down as $\mathfrak  R_T=\frac{\mathbb E [\lambda_T]-\lambda_T}{\beta \sqrt T}$. Hence
    \begin{align*}
    \mathbb E [\mathfrak  R_T^2] &= \mathbb E \left[\left(\frac{\mathbb E [\lambda_T]-\lambda_T}{\beta \sqrt T}\right)^2\right],\\
    &=\frac{\mathbb E [(\lambda_T -\mathbb E [\lambda_T])^2]}{\beta^2 T},\\
    &=\frac{\textrm{Var}(\lambda_T)}{\beta^2 T},\quad \text{since the second moment of $\lambda$ is bounded (according to Lemma $\ref{lemma:secondorder}$)},\\
    &=O\left(\frac{1}{T}\right).
    \end{align*}
    \textbf{Case 2 : the kernel is an Erlang function}\\
    In this case the squared remainder becomes bounded by:
    \begin{align*}
    \mathbb E [\mathfrak  R_T^2]&=\mathbb E \left[\left(\frac{\mathbb E [\lambda_T]-\lambda_T}{\beta \sqrt T}+\frac{\mathbb E [\xi_T]-\xi_T}{\beta^2 \sqrt T}\right)^2\right],\\
    &\leq 2 \mathbb E \left[\left(\frac{\mathbb E [\lambda_T]-\lambda_T}{\beta \sqrt T}\right)^2 + \left(\frac{\mathbb E [\xi_T]-\xi_T}{\beta^2 \sqrt T}\right)^2\right] ,\quad \text {since $2ab \leq a^2 +b^2$},\\
    &\leq \frac{2 \textrm{Var}(\lambda_T)}{\beta^2 T}+ \frac{2 \textrm{Var}(\xi_T)}{\beta^2 T},\\
    &\leq O\left(\frac{1}{T}\right)\quad \text{since the second moment of $\xi$ is bounded (according to the proof of Lemma $\ref{lemma:secondorder}$)}.
    \end{align*}
    Hence, in both cases one has 
    $$d_W\left(\frac {Y_T}{\gamma},\frac{G}{\gamma}\right) \leq \frac{C}{\sqrt T}.$$
   And since $d_W(\frac {Y_T}{\gamma},\frac{G}{\gamma})=\frac{1}{\gamma}d_W(Y_T,G)$ (replace $f(x)$ with $f_\gamma(x)=\frac{f(\gamma x)}{\gamma}$ in the Wasserstein distance definition) we conclude that there is a positive constant $C_{\alpha,\beta}$ that does not depend on $T$ such that
    $$d_W(Y_T,G)\leq \frac{C_{\alpha,\beta}}{\sqrt T}.$$
\end{proof}

{\bf  Conclusion} : In this paper we have  computed  Berry-Ess\'een bounds associated to Central Limit Theorems for the compound Hawkes process,  using a Mallavin-Stein approach, also known as Nourdin-Peccati's approach. Since a compound Hawkes process  is a natural model for the cumulative loss process of an insurance portfolio exhibiting self-exciting features,  such bounds are of particular interest for the ruin theory. This is a work in progress. 

\section{Appendix}
\label{section:teclemma}

Most of the estimates presented in this section are focused on the Hawkes process and its intensity. As the reader will figure out, the $x$-variable (representing the role of the random variables $(Y_i)_{i\geq 1}$) will be factored out of the computations. Hence, the technology is focused on the Hawkes process together with its intensity as if $\nu(dx)=\delta_1(dx)$.\\
We recall the original system $(X,H,\lambda)$ as (\ref{eq:H}) and $(\hat X^t, \hat H^t,\hat \lambda^t)$ describing the Malliavin derivative recalled in Notation \ref{notation:systemhat}.

\subsection{General estimates}

Throughout this section, we assume that Assumptions \ref{assumptionY1}, \ref{assumptionPhi1} and \ref{assumptionPhi2} are in force. 
\begin{lemma}
\label{lemma:cvgsigma}
For any $T>0$, and recalling that $\sigma^2 = \frac{\mu}{1-\|\phi\|_1}$,
$$\left|\frac{\sigma^2}{\vartheta^2}-\frac{1}{T}\int_0^T\mathbb E [\lambda_t]dt\right| =O\left(\frac{1}{T}\right).$$
\end{lemma}

\begin{proof}
According to \cite[Theorem 2]{Bacry_et_al_2013}, under Assumptions \ref{assumptionPhi1} and \ref{assumptionPhi2}, 
$$\mathbb E [H_t]=\int_0^t \mathbb E[\lambda_s]ds=\mu t +\mu \int_0^t\psi(t-s)sds,$$
where $\psi$ is defined by (\ref{eq:Psi}). Recall that 
$$\int_0^{+\infty} \psi (t) dt =\frac{\|\phi\|_1}{1-\|\phi\|_1}.$$
As $\sigma^2=\lim_{T\to +\infty} \frac{\E[H_T]}{T} = \lim_{T\to +\infty} \frac{\int_0^T \E[\lambda_t] dt}{T}$, the following  computation gives that $\sigma^2 = \frac{\mu}{1-\|\phi\|_1}$. Indeed, 
\begin{align*}
    \frac{1}{T}\int_0^T\mathbb E [\lambda_t]dt -\sigma^2 &=\mu  +\frac{\mu}{T} \int_0^T\psi(T-s)sds -\sigma^2,\\
    &=\mu  +\frac{\mu}{T} \int_0^T\psi(s)(T-s)ds-\sigma^2,\\
    &=\mu +\mu \int_0^T\psi(s)ds -\frac{\mu}{T}\int_0^Ts\psi(s)ds -\sigma^2,\\
    &=\mu +\mu \int_0^{+\infty}\psi(s)ds -\mu\int_T^{+\infty}\psi(s)ds-\frac{\mu}{T}\int_0^Ts\psi(s)ds -\sigma^2,\\
    &= \left(\frac{\mu}{1-\|\phi\|_1} - \sigma^2\right)  -\mu\int_T^{+\infty}\psi(s)ds-\frac{\mu}{T}\int_0^Ts\psi(s)ds.
\end{align*}
Following \cite[Lemma 5]{Bacry_et_al_2013}, $\int_0^{+\infty}s\psi(s)ds < +\infty$ which entails that $\frac{\mu}{T}\int_0^Ts\psi(s)ds=O\left(\frac{1}{T}\right)$. Concerning the other term, one can note that 
\begin{align*}
    \int_T^{+\infty}\psi(s)ds&\leq \int_T^{+\infty}\psi(s)\frac{s}{T}ds
    \leq\frac{1}{T}\int_0^{+\infty}s\psi(s) ds.
\end{align*}
These estimates conclude the proof.
\end{proof}

\begin{lemma}
\label{lemma:momentDH}
We make use of the notation of Proposition \ref{prop:DescDH}. Let $T>0$ and $0 \leq s \leq t \leq T$. The following estimate holds 
$$ \int_t^T \mathbb E_t[\hat \lambda_s^t]ds \leq \|\phi\|_1(1+ \|\psi\|_1).$$
\end{lemma}

\begin{proof}
Recall that for $t\geq 0$, $\E_t[\cdot]$ stands for the conditional expectation $\E[\cdot\vert \mathcal F_t^N]$. 
Taking the conditional expectation in Equation (\ref{eq:DH}) for $0\leq t \leq s$ leads to
$$\mathbb E _t[\hat \lambda_s^t]=\phi(s-t)+\int_{(t,s)}\phi (s-u)\mathbb E_t[\hat \lambda^t_u]du.$$
For $t\geq 0$ set $f_t(s):=\mathds 1_{s\geq t}\mathbb E _t[\hat \lambda_s^t]$ and $\phi_t(s):=\mathds 1_{s\geq t} \phi(s-t)$, the last equation becomes
$$f_t(s)=\phi_t(s)+\int_0^s \phi(s-u)f_t(u)du.$$
A straightforward application of \cite[Lemma 3]{Bacry_et_al_2013} yields
$$f_t(s)=\phi_t(s)+\int_0^s \psi(s-u)\phi_t(u)du,$$
and by integrating between $t$ and $T$ one obtains
\begin{align*}
    \int_t^T \mathbb E_t[\hat \lambda_s^t]ds &= \int _t^T f_t(s)ds,\\
    &=\int _t^T \phi_t(s) ds + \int _t^T \int_0^s\psi(s-u)\phi_t(u)du ds,\\
    &=\int _t^T \phi_t(s) ds + \int _t^T \int_t^s\psi(s-u)\phi_t(u)du ds,\\
    &=\int _t^T \phi_t(s) ds + \int _t^T \int_t^T\mathds 1_{s\geq u}\psi(s-u)\phi_t(u)du ds.
\end{align*}
Since all the involved functions are positive, we use Fubini's theorem and exchange the integrals to get
\begin{align*}
    \int_t^T \mathbb E_t[\hat \lambda_s^t]ds &=\int _t^T \phi_t(s) ds + \int _t^T \int_u^T\psi(s-u)ds\phi_t(u)du\\
    &\leq \|\phi\|_1+\int _t^T \|\psi\|_1 \phi_t(u)du\\
    &\leq \|\phi\|_1(1+ \|\psi\|_1).
\end{align*}
\end{proof}

\subsection{Specific estimates for the exponential and the Erlang's kernels}\label{sec:kernel}

\subsubsection{Definition and some properties of the exponential and the Erlang Hawkes processes}

\begin{definition}[Exponential Hawkes process]
A counting process $H$ as in Definition \ref{def:standardHawkes} is referred to 
\begin{itemize}
\item[(i)] an exponential Hawkes process if there exist $(\alpha,\beta) \in \mathbb{R}_+^2$ such that  
\begin{equation}
\label{eq:expokernel}
\alpha < \beta \quad \textrm{ and } \quad \Phi(u):=\alpha e^{-\beta u}, \quad u \in \mathbb{R}_+. 
\end{equation}
\item[(ii)] an Erlang Hawkes process if there exist $(\alpha,\beta) \in \mathbb{R}_+^2$ such that 
\begin{equation}
\label{eq:Erlangkernel}
\alpha < \beta^2 \quad \textrm{ and } \quad \Phi(u):=\alpha u e^{-\beta u}, \quad u \in \mathbb{R}_+. 
\end{equation}
\end{itemize}
\end{definition}

\begin{remark}
Obviously, the two kernels above ((\ref{eq:expokernel})-(\ref{eq:Erlangkernel})) satisfy Assumptions \ref{assumptionPhi1} and (\ref{assumptionPhi2}).
\end{remark}
\begin{prop}
	\label{prop:Dynkin}
	Let $H$ be a Hawkes process and let $\lambda$ be its intensity. 
	\begin{itemize}
		\item [(i)] If $H$ is an exponential Hawkes process, then $(\lambda_t)_{t\in \mathbb R_+}$ is a Markov process that satisfies the following Dynkin formula for each function $g\in \mathcal C^1$ and for each $t\le T$:
		$$\E[g(\lambda_T)|\mathcal F_t]=g(\lambda_t)+\E \big [ \int_t^T (\mathcal D g)(\lambda_s)ds|\mathcal F_t\big],$$
		where $\mathcal D $ is the infinitesimal generator:
		$$\mathcal D g (\lambda):=\beta (\mu -\lambda)g'(\lambda)+\lambda \big ( g(\lambda+\alpha)-g(\lambda)\big),$$
		whenever these expectations are finite.
		\item [(ii)] If $H$ is an Erlang Hawkes process, then $(\lambda_t,\xi_t)_{t\in \mathbb R_+}$ is a Markov process, where $\xi_t =\int_0^t \alpha e^{-\beta (t-s)}dH_s$ is an auxiliary process. In this case, $(\lambda_t,\xi_t)_{t\in \mathbb R_+}$ satisfies this Dynkin formula for each function $g\in \mathcal C^1$ and for each $t\le T$
		$$\E[g(\lambda_T,\xi_T)|\mathcal F_t]=g(\lambda_t,\xi_t)+\E \big [ \int_t^T (\mathcal D g)(\lambda_s,\xi_s)ds|\mathcal F_t\big],$$
		where $\mathcal D $ is the infinitesimal generator
		$$\mathcal D g (\lambda,\xi):=\big (\xi+\beta (\mu -\lambda)\big)\partial _\lambda g(\lambda,\xi)-\beta \xi \partial_\xi g(\lambda,\xi)+\lambda \big ( g(\lambda,\xi+\alpha)-g(\lambda,\xi)\big),$$
		whenever these expectations are finite.
	\end{itemize}	
\end{prop}
\begin{proof}
	The exponential case comes from Proposition 2.1 in \cite{errais2010affine} and the Erlang case can be found in Proposition 1 in \cite{duarte2019stability}. 
\end{proof}

\begin{lemma}
	\label{lemma:secondorder}
	Assume that $H$ is an exponential or an Erlang Hawkes process (that is  Condition (\ref{eq:expokernel}) or (\ref{eq:Erlangkernel}) is in force for some parameters $\alpha, \beta$).
	Then there is a positive constant $C$ (depending only on the parameters $\alpha$ and $\beta$) such that for any $t\geq 0$
	$$\E [\lambda_t^2] \leq C.$$ 
\end{lemma}
\begin{proof}
We treat the two cases (\ref{eq:expokernel}) or (\ref{eq:Erlangkernel}) separately. 
	\begin{itemize}
		\item If $\Phi(u)=\alpha e^{-\beta u}$, $\alpha < \beta$\\
		According to \cite {errais2010affine}, the intensity is a Markov process that satisfies the following Dynkin formula:
		$$\frac{d}{dt} \E[g(\lambda_t)]=\E[\beta (\mu-\lambda_t)g'(\lambda_t)+\lambda_t \big(g(\lambda_t+\alpha)-g(\lambda_t) \big)].$$
		
	This formula can be used to have the variance  of the intensity $\lambda_t$ and an upper bound explicitly.
	\item  If $\Phi(u)=\alpha u e^{-\beta u}$\\
	Even though the intensity is no longer Markov it is possible to see it as a part of a 'Markov cascade' (cf \cite{duarte2019stability}) with an auxiliary process $(\xi_t)_{t\geq 0}$. The Dynkin formula for the vector process $(\lambda_t,\xi_t)_{t\geq 0}$ is
	$$\frac{d}{dt} \E[g(\lambda_t,\xi_t)]=\E\big[\big (\xi_t +\beta (\mu-\lambda_t)\big)\partial _{\lambda }g(\lambda_t,\xi_t) -\beta \xi_t \partial _{\xi}g(\lambda_t,\xi_t)+\lambda_t \big( g(\lambda_t,\xi_t+\alpha)-g(\lambda_t,\xi_t)\big)\big].$$
	After taking $g(\lambda,\xi)=\lambda^2$, $g(\lambda,\xi)=\xi^2$ and $g(\lambda,\xi)=\lambda \xi$, we obtain the following system
	\begin{equation}
	\systeme{
		\partial_t \E[\lambda^2_t]=-2\beta \E[\lambda^2_t]+2\E[\lambda_t \xi_t]+2\beta\mu\E[\lambda_t],\\
		\partial_t \E[\xi^2_t]=-2\beta \E[\xi^2_t] +2\alpha \E[\lambda_t \xi_t]+\alpha^2 \E[\lambda_t],\\
		\partial_t \E[\lambda_t \xi_t]=\alpha \E[\lambda^2_t] +\E[\xi^2_t]-2\beta \E[\lambda_t \xi_t]+\beta \mu \E[\xi_t]
	}.
	\end{equation}
	This is a linear system that can be put under the following matrix form:
	\begin{equation*}
	\frac{d}{dt}
	\begin{pmatrix}
	\E [\lambda_t^2]\\
	\E [\xi_t^2]\\
	\E[\lambda_t \xi_t]\\
	\end{pmatrix}
	=
	\begin{pmatrix}
	-2\beta & 0 & 2\\
	0 & -2 \beta & 2\alpha\\
	\alpha & 1 & -2 \beta\\
	\end{pmatrix}
	\begin{pmatrix}
	\E [\lambda_t^2]\\
	\E [\xi_t^2]\\
	\E[\lambda_t \xi_t]\\
	\end{pmatrix}
	+\begin{pmatrix}
	2 \beta \mu \E [\lambda_t]\\
	\alpha^2 \E [\lambda_t]\\
	\beta \mu \E [\xi_t] 
	\end{pmatrix}.
	\end{equation*}
	The matrix has three distinct negative eigenvalues: $v=-2\beta$ and $v_{\pm}=-2\beta \pm 2\sqrt{\alpha}$.\\
	Since $\E[\lambda_t]$ and $\E[\xi_t]$ are both bounded by a constant, we conclude that there is $C>0$ such that
	$$\E[\lambda_t^2]\leq C.$$
	\end{itemize}
\end {proof}

\begin{lemma}
	\label{lemma:A13}
	Assume that $\Phi(u)=\alpha e^{-\beta u}$ (with $\alpha < \beta $) or $\Phi(u)=\alpha u e^{-\beta u}$ (with $\alpha < \beta^2 $).\\
	We remind that $A_{1,3}=\frac{1}{T} \E\left[\left| \int_0^T \lambda_t \hat M_T^t dt \right| \right]$ and that $\hat M_\cdot^t$ is defined in Notation \ref{notation:systemhat}. For any T>0, we have
	$$ A_{1,3}=O\left(\frac{1}{\sqrt T}\right). $$
\end{lemma}
\begin{proof}
We proceed in three steps. First we provide a general estimate for term $A_{1,3}$ for a general kernel $\Phi$. Then we make use of this estimate in the two particular cases of exponential and Erlang kernels (\ref{eq:expokernel})-(\ref{eq:Erlangkernel}).\\\\
\textbf{Step 1 : a general estimate}\\
In the following lines we provide an estimate for term $A_{1,3}$ for a general kernel $\Phi$ enjoying Assumptions \ref{assumptionPhi1} and \ref{assumptionPhi2}.
	The idea is to bound the term $A_{1,3}$ as tightly as possible. To do so, we start with a Cauchy-Schwarz inequality to get
	$$A_{1,3} \leq \frac{1}{T} \E\left[\left| \int_0^T \lambda_t \hat M_T^t dt \right|^2\right]^{\frac{1}{2}}. $$
	From now on we are interested in the term $\E\left[\left| \int_0^T \lambda_t \hat M_T^t dt \right|^2\right].$ By expanding the square one obtains
	\begin{align*}
	\E\left[\left|\int_0^T \lambda_t \hat M_T^t dt\right|^2\right]
	&=2 \E \left[\int_0^T \int_0^t \lambda_t \lambda_s \hat M_T^t \hat M_T^s\right]ds dt\\
	&=2 \int_0^T \int_0^t \E[\lambda_t\lambda_s \E_t[\hat M_T^t \hat M_T^s]] ds dt\\
	&=2 \int_0^T \int_0^t \E[\lambda_t\lambda_s \E_t[\hat M_T^t (\hat M_T^s-\hat M_t^s)]] ds dt \\
	&=2 \int_0^T \int_0^t \E[\lambda_t\lambda_s \E_t\left[[\hat  M^t, (\hat M^s-\hat M_t^s)]_T]\right] ds dt,
	\end{align*}
	where we have used the fact that for any $t$, $\hat M_\cdot^t$ is a martingale with $\hat M_t^t=0$.
	Hence
\begin{align*}
&\E_t[[\hat  M^t, (\hat M^s-\hat M_t^s)]_T]\\
&=\E_t\left[ \sum _{t<u\leq T} \Delta _u \hat M ^t \Delta _u \hat M ^s \right] \\
&=\E_t\left[ \int_{(t,T] \times E} x \textbf{1}_{\{\theta \leq \hat \lambda_u^t\}} N(du,d\theta,dx) \int_{(t,T] \times E} x \textbf{1}_{\{\theta \leq \hat \lambda_u^s\}} N(du,d\theta,dx) \right] \\
&=\E_t\left[ \int_{(t,T] \times E} x^2 \textbf{1}_{\{\theta \leq \hat \lambda_u^t\}} \textbf{1}_{\{\theta \leq \hat \lambda_u^s\}} du d\theta \nu(dx) \right] \\
&= \vartheta^2 \; \E_t\left [\int _t^T \min (\hat \lambda^t_u,\hat \lambda^s_u)du \right].
\end{align*}
	The last entity can be bounded as follows
	$$\E_t[[\hat M^t, (\hat M^s-\hat M_t^s)]_T] \leq  \vartheta^2 \int _t^T \min (\E_t[\hat\lambda^t_u], \E_t[\hat\lambda^s_u])du.$$
	Since the intensity $\lambda$ of a Hawkes process is positive, multiplying by it preserves the inequality, leading to 
	\begin{align}
	\label{eq:A13estgeneral}
	A_{1,3} &\leq \frac{1}{T} \E\left[\left| \int_0^T \lambda_t \hat M_T^t dt \right|^2\right]^{\frac{1}{2}} \nonumber\\
	&\leq \vartheta^2 \frac{\sqrt{2}}{\sqrt{T}}  \left(\frac{1}{T}\int_0^T \int_0^t \E\left[\lambda_t\lambda_s \int _t^T \min (\E_t[\hat\lambda^t_u], \E_t[\hat\lambda^s_u])du\right]ds dt\right)^{1/2}\\
	&=: \vartheta^2 \frac{\sqrt{2}}{\sqrt{T}}  \left(\frac{1}{T}I_T\right)^{1/2},
	\end{align}
 with obvious notation. 
	The rest of the proof consists in specifying estimates of Quantity $I_T$ for the two particular Hawkes processes under interest.\\\\ 
\textbf{Step 2 : the exponential case} \\
We assume an exponential kernel  (\ref{eq:expokernel})  $\Phi(u)=\alpha e^{-\beta u} $ with $0 < \alpha < \beta$.\\
The main advantage of the Markov framework of the exponential kernel is the fact that many formulae for $\lambda$ and $\hat \lambda$ are known explicitly. In fact, given a starting time $t$, $\hat \lambda^t$ is defined for $u\geq t$ and it satisfies the following SDE (using once again Notation \ref{notation:systemhat})
$$d\hat \lambda_u^t= (\alpha-\beta)\hat \lambda_u^t du +\alpha d\mathcal M^t_u,$$
with $\mathcal M^t_u= \hat H^t_u - \int_t^u \hat\lambda^t_s ds$,
which yields after taking the initial conditions into account
$$\E_t [\hat\lambda^t_u]=\alpha e^{(\alpha-\beta)(u-t)} \quad \textrm{and} \quad \E_t [\hat\lambda^s_u]=\hat\lambda^s_t e^{(\alpha-\beta)(u-t)}.$$
The bound on $I_T$ becomes (from now on everything is written up to a positive multiplicative constant $C$ depending on $\alpha, \beta$ that may differ from line to line)
\begin{align*}
I_T &\leq C \int_0^T \int_0^t \E[\lambda_t \lambda_s \int _t^T \min (\alpha,\hat\lambda^s_t)e^{(\alpha-\beta)(u-t)}du]ds dt,\\
&=C\int_0^T \int_0^t \E[\lambda_t \lambda_s\min (\alpha,\hat\lambda^s_t)]\int _t^T e^{(\alpha-\beta)(u-t)}du ds dt,\\
&=C\int_0^T \int_0^t \E[\lambda_t \lambda_s\min (\alpha,\hat\lambda^s_t)] ds(1-e^{(\alpha-\beta)(T-t)})dt,\\
&=C\int_0^T \int_0^t \E[\lambda_s \E_s[\lambda_t\min (\alpha,\hat\lambda^s_t)]] ds(1-e^{(\alpha-\beta)(T-t)})dt,\\
&\leq C\int_0^T \int_0^t \E[\lambda_s \E_s[\lambda_t\hat\lambda^s_t]] ds(1-e^{(\alpha-\beta)(T-t)})dt.
\end{align*}
Since $\hat \lambda_t^s$ starts at $s$, it is independent from $\mathcal F^N_s$  thus $\E_s[\lambda_t\hat\lambda^s_t]=\E_s[\lambda_t]\E[\hat\lambda^s_t]= \E_s[\lambda_t] \alpha e^{(\alpha-\beta)(t-s)}$.\\
Hence the Cauchy-Schwarz inequality together with Lemma \ref{lemma:secondorder} entail :
\begin{align*}
I_T &\leq C \int_0^T \int_0^t \E[\lambda_s \E_s[\lambda_t]] e^{(\alpha-\beta)(t-s)}ds (1-e^{(\alpha-\beta)(T-t)})dt,\\
&= C \int_0^T \int_0^t \E[\lambda_s \lambda_t] e^{(\alpha-\beta)(t-s)}ds (1-e^{(\alpha-\beta)(T-t)})dt,\\
&\leq C \int_0^T \int_0^t \E[\lambda_s^2]^{\frac{1}{2}} \E[\lambda_t^2]^{\frac{1}{2}} e^{(\alpha-\beta)(t-s)}ds (1-e^{(\alpha-\beta)(T-t)})dt, \quad \textrm{Cauchy-Schwarz}\\
&\leq C \int_0^T \int_0^te^{(\alpha-\beta)(t-s)}ds (1-e^{(\alpha-\beta)(T-t)})dt,  \quad {\textrm{using Lemma \ref{lemma:secondorder}}}\\
&= C \int_0^T (1-e^{(\alpha-\beta)t})(1-e^{(\alpha-\beta)(T-t)})dt,\\
&= C \int_0^T 1 -e^{(\alpha-\beta)(T-t)}-e^{(\alpha-\beta)t} +e^{(\alpha-\beta)T}dt,\\
&\leq C (T +e^{(\alpha-\beta)T}-1 + Te^{(\alpha-\beta)T}) \quad \textrm{since $\alpha-\beta <0$},\\
&\leq C \cdot T.
\end{align*}
Finally  using (\ref{eq:A13estgeneral})  we conclude that \,
$A_{1,3} = O \left(\frac{1}{\sqrt {T}}\right).$\\

\noindent \textbf{Step 3 : the Erlang case} \\
We assume an Erlang kernel  (\ref{eq:Erlangkernel}) $\Phi(u)=\alpha u e^{-\beta u}$ with $0 < \alpha < \beta^2$.\\
Even though the intensity $\lambda$ is no longer a Markov process, it is possible to "Markovize" it by taking an auxiliary process $\xi$ into account.\\
In the case of the shifted vanishing process $\hat \lambda^s$, $\hat \xi ^s_u=\alpha e^{-\beta (u-s)}+\alpha \int _s^u  e^{-\beta (u-v)} d\hat H^s_v, $ and the process $(\hat \lambda^s,\hat \xi^s)$ follows the SDE:
\begin{equation*} 
\systeme{
	d\hat \lambda_u^s=-\beta \hat \lambda_u^s du+\hat \xi _u^s du,
	d\hat \xi _u^s=-\beta \hat \xi _u^s du +\alpha d\hat H^s_u.
}
\end{equation*}
This  yields after solving the system
\begin{align*}
\begin{pmatrix}
\hat \lambda_u^s\\
\hat \xi_u^s
\end{pmatrix}
=&
\frac{\sqrt \alpha}{2}
\big(
(\int_s^u e^{( \sqrt \alpha-\beta)(s-v)}d\hat{ \mathcal M^s_v}+1)e^{( \sqrt \alpha-\beta)(u-s)}
\begin{pmatrix}
1\\
\sqrt \alpha
\end{pmatrix}\\
&-
(\int_s^u e^{-( \sqrt \alpha+\beta)(s-v)}d\hat{ \mathcal M^s_v}+1)e^{-( \sqrt \alpha+\beta)(u-s)}
\begin{pmatrix}
1\\
-\sqrt \alpha
\end{pmatrix}
\big).
\end{align*}
And finally
\begin{equation}
\label {equation:expectation}
\begin{pmatrix}
\mathbb E_t[\hat \lambda_u^s]\\
\mathbb E_t[\hat \xi_u^s]
\end{pmatrix}
=
\frac{1}{2}(\hat \lambda_t^s+\frac{\hat \xi_t^s}{\sqrt \alpha})e^{(\sqrt \alpha-\beta)(u-t)}
\begin{pmatrix}
1\\
\sqrt \alpha
\end{pmatrix}
+
\frac{1}{2}(\hat \lambda_t^s-\frac{\hat \xi_t^s}{\sqrt \alpha})e^{-(\sqrt \alpha+\beta)(u-t)}
\begin{pmatrix}
1\\
-\sqrt \alpha
\end{pmatrix}
.
\end{equation}
Once again we make use of the general estimate (\ref{eq:A13estgeneral}) and estimate the quantity 
$$I_T = \int_0^T \int_0^t \E[\lambda_t\lambda_s \int _t^T \min (\E_t[\hat\lambda^t_u], \E_t[\hat\lambda^s_u])du]ds dt,$$
We have
$$I_T\leq \int_0^T \int_0^t \E[\lambda_t\lambda_s \int _t^T  \E_t[\hat\lambda^s_u]du] ds dt.$$
We chose $\E_t[\hat\lambda^s_u]$ on purpose, because since the process starts earlier it has more chances of vanishing at time $t$. 
\begin{align*}
I_T \leq & C\int_0^T \int_0^t \E[\lambda_t\lambda_s \int _t^T \frac{1}{2}(\hat \lambda_t^s+\frac{\hat \xi_t^s}{\sqrt \alpha})e^{(\sqrt \alpha-\beta)(u-t)}+\frac{1}{2}(\hat \lambda_t^s-\frac{\hat \xi_t^s}{\sqrt \alpha})e^{-(\sqrt \alpha+\beta)(u-t)} du ds dt,\\
=&C\int_0^T \int_0^t \E[\lambda_t\lambda_s (\hat \lambda_t^s+\frac{\hat \xi_t^s}{\sqrt \alpha})]\int _t^T e^{(\sqrt \alpha-\beta)(u-t)} du ds dt \\ &+C \int_0^T \int_0^t \E[\lambda_t\lambda_s (\hat \lambda_t^s-\frac{\hat \xi_t^s}{\sqrt \alpha})] \int _t^T e^{-(\sqrt \alpha+\beta)(u-t)} du ds dt,\\
\leq& C\int_0^T (1-e^{(\sqrt \alpha-\beta)(T-t)}) \int_0^t \E[\lambda_t\lambda_s (\hat \lambda_t^s+\frac{\hat \xi_t^s}{\sqrt \alpha})] ds dt \\&+C\int_0^T (1-e^{-(\sqrt \alpha+\beta)(T-t)}) \int_0^t \E[\lambda_t\lambda_s (\hat \lambda_t^s-\frac{\hat \xi_t^s}{\sqrt \alpha})] ds dt,\\
:=& C(I_T^+ + I_T^-).
\end{align*}
Like in the Markov case, $\hat \lambda ^s$ and $\hat \xi^s$ both start at time $s$, thus $\E_s [\lambda_t \hat \lambda ^s_t]=\E_s[\lambda_t] \E[\hat \lambda ^s_t]$ and the same holds for $\hat \xi^s$ as well. 
In addition, since $\E_s[\hat \lambda_t^s ]=\E[\hat \lambda_t^s ]$ and since $ \hat \lambda_s^s=0$, Equation \ref{equation:expectation} yields
$$\mathbb E[\hat \lambda_t^s]=\frac{\sqrt \alpha}{2}(e^{( \sqrt \alpha-\beta)(t-s)} -e^{-( \sqrt \alpha+\beta)(t-s)}),$$
and
$$\frac{\E[\hat \xi^s_t]}{\sqrt \alpha}=\frac{\sqrt \alpha}{2}(e^{( \sqrt \alpha-\beta)(t-s)} +e^{-( \sqrt \alpha+\beta)(t-s)}).$$
Hence
\begin{align*}
I_T^+ &=\int_0^T (1-e^{(\sqrt \alpha-\beta)(T-t)}) \int_0^t \E[\lambda_s \E_s [\lambda_t (\hat \lambda_t^s+\frac{\hat \xi_t^s}{\sqrt \alpha})]] ds dt ,\\
&=\int_0^T (1-e^{(\sqrt \alpha-\beta)(T-t)}) \int_0^t \E[\lambda_s \E_s [\lambda_t] \E[\hat \lambda_t^s+\frac{\hat \xi_t^s}{\sqrt \alpha}]] ds dt,\\
&=\int_0^T (1-e^{(\sqrt \alpha-\beta)(T-t)})\int_0^t \E[\lambda_t \lambda_s] \sqrt \alpha e^{( \sqrt \alpha-\beta)(t-s)} ds dt.
\end{align*}
As for the exponential kernel, it is possible to have bounds on the intensity's second moment using Lemma \ref{lemma:secondorder}, which yields by the Cauchy-Schwarz inequality
$$\E[\lambda_t \lambda_s] \leq \E[\lambda_s^2]^{\frac{1}{2}} \E[\lambda_t^2]^{\frac{1}{2}} \leq C.$$
Finally
\begin{align*}
I_T^+ &\leq \int_0^T (1-e^{(\sqrt \alpha-\beta)(T-t)}) \int_0^t e^{( \sqrt \alpha-\beta)(t-s)} ds dt,\\
&\leq \int_0^T (1-e^{(\sqrt \alpha-\beta)(T-t)}) (1-e^{(\sqrt \alpha-\beta)t})dt,\\
&\leq \int_0^T 1-e^{(\sqrt \alpha-\beta)t} -e^{(\sqrt \alpha-\beta)(T-t)} + e^{(\sqrt \alpha-\beta)T}dt,\\
& = O(T).
\end{align*}
Following the same lines we get for $|I_T^-|$ :
\begin{align*}
I_T^- &=\int_0^T (1-e^{-(\sqrt \alpha+\beta)(T-t)}) \int_0^t \E[\lambda_s \E_s [\lambda_t (\hat \lambda_t^s-\frac{\hat \xi_t^s}{\sqrt \alpha})]] ds dt ,\\
&=\int_0^T (1-e^{-(\sqrt \alpha+\beta)(T-t)}) \int_0^t \E[\lambda_s \E_s [\lambda_t] \E[\hat \lambda_t^s-\frac{\hat \xi_t^s}{\sqrt \alpha}]] ds dt,\\
&=-\int_0^T (1-e^{-(\sqrt \alpha+\beta)(T-t)})\int_0^t \E[\lambda_t \lambda_s] \sqrt \alpha e^{-( \sqrt \alpha+\beta)(t-s)} ds dt.
\end{align*}
So $|I_T^-| = O(T)$, and since $I_T \leq I_T^+ + |I_T^-|$, we conclude using once again (\ref{eq:A13estgeneral}) that
$$A_{1,3} = O\left( \frac{1}{\sqrt {T}}\right).$$
\end{proof}

\begin{lemma}
	\label{lemma:alt} Set $Y_T=\frac{H_T-\int_0^T\E[\lambda_t]dt}{\sqrt T}$ and $F_T=\frac{H_T-\int_0^T \lambda_t dt}{\sqrt T}$.\\
	Assume that $\Phi(u)=\alpha e^{-\beta u}$ (with $\alpha < \beta $) or $\Phi(u)=\alpha u e^{-\beta u}$ (with $\alpha < \beta^2 $).\\ 
	Then $Y_T$ and $F_T$ satisfy the relation:
	$$\frac{Y_T}{\gamma}=F_T+\mathfrak R_T,$$
	where 
	
\begin{align*}
\gamma= \frac{1}{1-\|\Phi\|_1}=\begin{cases}
\frac{1}{1-\alpha/\beta}, \quad \text {if} \quad\Phi(u)=\alpha e^{-\beta u} \\
\frac{1}{1-\alpha/\beta^2}, \quad \text {if} \quad \Phi(u)=\alpha u e^{-\beta u} \\
\end{cases}
\end{align*}
and 
\begin{align*}
\mathfrak  R_T= \begin{cases} 
\frac{\mathbb E [\lambda_T]-\lambda_T}{\beta \sqrt T}, \quad \text {if} \quad\Phi(u)=\alpha e^{-\beta u} \\
\frac{\mathbb E [\lambda_T]-\lambda_T}{\beta \sqrt T}+\frac{\mathbb E [\xi_T]-\xi_T}{\beta^2 \sqrt T}, \quad\text {if} \quad \Phi(u)=\alpha u e^{-\beta u} \\
\end{cases}
\end{align*}
\end{lemma}
\begin{proof}
	We remind that the advantage of the exponential and Erlang kernels is the SDE. This SDE will be used to eliminate the integral in $Y_T$ and $F_T$ in order to obtain alternative expressions.\\\\
	\textbf{Case 1 : the kernel is an exponential function\\}
	We recall that $\|\Phi\|_1= \frac{\alpha}{\beta}$ and that the intensity is a solution to the following SDE:
	$$d\lambda_t=\beta (\mu -\lambda_t)dt +\alpha dH_t.$$ We start by integrating the SDE between $0$ and $T$:
	\begin{align*}
	\int _0^T d\lambda_t&=\int_0^T \beta(\mu-\lambda_t)dt+\alpha dH_t, \\
	\lambda_T -\mu &= \beta \mu T - \beta \int_0^T\lambda_t dt +\alpha H_T.
	\end{align*}
	After rearranging the terms, $F_T$ can be put under the form:
	$$F_T = \frac{\lambda_T-\mu +(\beta -\alpha)H_T-\beta \mu T}{\beta \sqrt T}.$$
	When it comes to $Y_T$, we start by taking the expected value of the SDE:
	\begin{align*}
	d\mathbb E [\lambda_t] &=\beta(\mu-\mathbb E [\lambda_t])dt+\alpha d \mathbb E[H_t],\\
	&=\beta(\mu-\mathbb E [\lambda_t])dt+\alpha \mathbb E[\lambda_t] dt,
	\end{align*}
	and after integrating with respect to time:
	\begin{align*}
	\mathbb E [\lambda_T] -\mu &= \beta \mu T +(\alpha-\beta) \int_0^T\mathbb E[\lambda_t] dt,\\
	&=\beta \mu T +(\alpha-\beta) (\int_0^T\mathbb E [\lambda_t] dt-H_T)+(\alpha-\beta) H_T.\\
	\end{align*}
	And finally:
	$$Y_T =\frac{\mathbb E[\lambda_T]-\mu +(\beta -\alpha)H_T-\beta \mu T}{(\beta-\alpha) \sqrt T}.$$
	
	These alternative expressions allow us to deduce a simple relation between $F_T$ and $Y_T$:
	$$\beta \sqrt T F_T - (\beta-\alpha)\sqrt T Y_T =\lambda_T -\mathbb E [\lambda_T],$$
	which is equivalent to:
	$$\frac{Y_T}{\gamma}=F_T + \mathfrak  R_T,$$
	where $\gamma =\frac{1}{1-\|\phi\|_1}= \frac{1}{1-\alpha /\beta}$ and $R_T =\frac{\mathbb E [\lambda_T]-\lambda_T}{\beta \sqrt T}$.\\\\
	\textbf{Case 2 : the kernel is an Erlang function \\}
	In this case $\|\Phi\|_1= \frac{\alpha}{\beta^2}$ and the vector $(\lambda_t,\xi_t)$ where $\xi_t=\int_{[0,t)}\alpha e^{-\beta(t-s)}dH_s$ follows satisfies the following SDE:
	 \begin{equation*}
	 \left\lbrace
	 \begin{array}{l}
	 d\lambda_t =  \xi_t dt +\beta (\mu - \lambda_t) dt, \\\\
	 d\xi_t = -\beta \xi_t dt +\alpha dH_T.
	 \end{array}
	 \right.
	 \end{equation*}
	 We integrate the system:
	 
	\begin{equation*}
	\left\lbrace
	\begin{array}{l}
	\lambda_T -\mu =  \int_0^T\xi_t dt +\beta \mu T -\beta \int_0^T \lambda_t dt, \\\\
	\xi_T = -\beta \int_0^T\xi_t dt +\alpha H_T.
	\end{array}
	\right.
	\end{equation*}
	We eliminate $\int_0^T \xi_t dt$ in the system to obtain:
	\begin{align*}
	\lambda_T-\mu &= \frac{\alpha}{\beta}H_T -\frac{1}{\beta} \xi_T +\mu \beta T -\beta \int_0^T \lambda_t dt,\\
	&=\frac{\alpha}{\beta}H_T -\frac{1}{\beta} \xi_T +\mu \beta T +\beta \sqrt T F_T -\beta H_T,
	\end{align*}
	and after re-arranging the terms:
	$$\beta \sqrt T F_T= \lambda_T-\mu +\frac{\beta ^2-\alpha}{\beta}H_T +\frac{1}{\beta}\xi_T -\mu \beta T.$$
	When it comes to $Y_T$ we take the expected value of the SDE:
	\begin{equation*}
	\left\lbrace
	\begin{array}{l}
	d\mathbb E[\lambda_t] = \big ( \mathbb E[\xi_t] +\beta (\mu -\mathbb E[\lambda_t])\big)dt, \\\\
	d\mathbb E[\xi_t] = -\beta \mathbb E [\xi_t] dt +\alpha \mathbb E[\lambda_t]dt.
	\end{array}
	\right.
	\end{equation*}
	After taking the integral and eliminating $\int_0^T \mathbb E [\xi_t]dt $  in the system:
	\begin{align*}
	\mathbb E [\lambda_T]-\mu &= -\frac{1}{\beta}\mathbb E [\xi_T] +\frac{\alpha}{\beta} \int_0^T \mathbb E [\lambda_t] dt +\beta \mu T -\beta \int_0^T \mathbb E [\lambda_t]dt,\\
	&=-\frac{1}{\beta}\mathbb E [\xi_T] +\beta \mu T +(\frac{\alpha}{\beta} -\beta)\int_0^T \mathbb E [\lambda_t] dt,\\
	&=-\frac{1}{\beta}\mathbb E [\xi_T] +\beta \mu T + \frac{\beta^2-\alpha}{\beta}\sqrt T Y_T+\frac{\alpha-\beta^2}{\beta} H_T,
	\end{align*}
	which yields after re-arranging the terms:
	$$\frac{\beta^2-\alpha}{\beta} \sqrt T Y_T= \mathbb E [\lambda_T] -\mu +\frac{\beta^2-\alpha}{\beta}H_T +\frac{1}{\beta}\mathbb E[\xi_T] -\mu\beta T.$$
	Combining these expressions on $F_T$ and $Y_T$ yields the following relation:
	$$\beta \sqrt T F_T-\frac{\beta^2-\alpha}{\beta} \sqrt T Y_T= \lambda_T -\mathbb E [\lambda_T] +\frac{1}{\beta} (\xi_T -\mathbb E [\xi_T]),$$
	which is equivalent to:
	$$\frac{Y_T}{\gamma}=F_T+\mathfrak  R_T$$
	with $\gamma=\frac{1}{1-\|\phi\|_1}=\frac{1}{1-\alpha/\beta^2}$ and $R_T=\frac{\mathbb E [\lambda_T]-\lambda_T}{\beta \sqrt T}+\frac{\mathbb E [\xi_T]-\xi_T}{\beta^2 \sqrt T}.$
\end{proof}


\end{document}